\documentclass[12pt,oneside,english]{amsart}
\usepackage{times}
\usepackage[T1]{fontenc}
\usepackage[latin1]{inputenc}
\usepackage{amssymb}

\makeatletter
 \theoremstyle{plain}
\newtheorem{thm}{Theorem}[section]
  \theoremstyle{definition}
  \newtheorem{defn}[thm]{Definition}
  \theoremstyle{plain}
  \newtheorem{prop}[thm]{Proposition}
 \theoremstyle{definition}
  \newtheorem{example}[thm]{Example}
  \theoremstyle{remark}
  \newtheorem{rem}[thm]{Remark}
  \theoremstyle{plain}
  \newtheorem{cor}[thm]{Corollary}
  \theoremstyle{plain}
  \newtheorem{lem}[thm]{Lemma}
  \theoremstyle{remark}
  \newtheorem*{acknowledgement*}{Acknowledgement}

\newcommand{\R}{{\mathbb R}}
\newcommand{\e}{\epsilon}

\newcommand{\cl}{\mbox{\rm cl}\,}
\newcommand{\calm}{\mbox{\rm calm}\,}
\newcommand{\calmx}{\mbox{\rm calm}\,}

\newcommand{\gph}{\mbox{\rm gph}\,}
\newcommand{\lip}{\mbox{\rm lip}\,}
\newcommand{\lipx}{\mbox{\rm lip}\,}
\newcommand{\lipi}{\mbox{\rm lip}\,}

\newcommand{\A}{X}

\newcommand{\rt}{\rightarrow}
\newenvironment{myequation}{\setcounter{equation}{\value{thm}}
   \begin{equation}}{\addtocounter{thm}{1}\end{equation}}

\newcommand{\bmye}{\begin{myequation}}
\newcommand{\emye}{\end{myequation}}

\usepackage{babel}
\makeatother
\begin{document}

\title{Lipschitz behavior of the robust regularization}

\author{Adrian S. Lewis and C.H. Jeffrey Pang }

\address{Adrian S. Lewis\\
School of Operations Research and Information Engineering \\
Cornell University\\
Ithaca, NY, 14853.}

\email{aslewis@orie.cornell.edu}

\address{C. H. Jeffrey Pang\\
Center for Applied Mathematics\\
Cornell University\\
Ithaca, NY, 14853.}

\email{cp229@cornell.edu}

\date{\today}

\begin{abstract}
To minimize or upper-bound the value of a function ``robustly'', we
might instead minimize or upper-bound the ``$\epsilon$-robust regularization'',
defined as the map from a point to the maximum value of the function
within an $\epsilon$-radius. This regularization may be easy to compute:
convex quadratics lead to semidefinite-representable regularizations,
for example, and the spectral radius of a matrix leads to pseudospectral
computations. For favorable classes of functions, we show that the
robust regularization is Lipschitz around any given point, for all
small $\epsilon>0$, even if the original function is nonlipschitz
(like the spectral radius). One such favorable class consists of the
semi-algebraic functions. Such functions have graphs that are finite
unions of sets defined by finitely-many polynomial inequalities, and
are commonly encountered in applications. 
\end{abstract}
\maketitle
\tableofcontents{}



\section{\label{sec:Introduction}Introduction}

In the implementation of the optimal solution of an optimization model,
one is not only concerned with the minimizer of the optimization model,
but how numerical errors and perturbations in the problem description
and implementation can affect the solution. We might therefore try
to solve an optimization model in a robust manner. The issues of robust
optimization, particularly in the case of linear and quadratic programming,
are documented in \cite{BtN}. 

A formal way to address robustness is to consider the {}``robust
regularization'' \cite{L02}. The notation {}``$\rightrightarrows$''
denotes a set-valued map. That is, if $F:X\rightrightarrows Y$ and
$x\in X$, then $F\left(x\right)$ is a subset of $Y$.

\begin{defn}
\label{def:robust-regularization}For $\epsilon>0$ and $F:X\rightarrow\mathbb{R}^{m}$,
where $X\subset\mathbb{R}^{n}$, the \emph{set-valued robust regularization}
$F_{\epsilon}:X\rightrightarrows\mathbb{R}^{m}$ is defined as \[
F_{\epsilon}\left(x\right):=\left\{ F\left(x+e\right)\mid\left|e\right|\leq\epsilon,x+e\in X\right\} .\]
For the particular case of a real-valued function $f:X\rightarrow\ \mathbb{R}$,
we define the \emph{robust regularization} $\bar{f}_{\epsilon}:X\rightarrow\mathbb{R}$
of $f$ by \begin{eqnarray*}
\bar{f}_{\epsilon}\left(x\right) & := & \sup\left\{ y\in f_{\epsilon}\left(x\right)\right\} \\
 & = & \sup\left\{ y\mid\exists x^{\prime}\in X\mbox{ such that }f\left(x^{\prime}\right)=y\mbox{ and }\left|x^{\prime}-x\right|\leq\epsilon\right\} .\end{eqnarray*}

\end{defn}
In this paper, we restrict our attention to the real-valued robust
regularization $\bar{f}_{\epsilon}:X\rightarrow\mathbb{R}$. The use
of set-valued analysis is restricted to Section \ref{sec:general-r-r}.

The minimizer of the robust regularization protects against small
perturbations better, and might be a better solution to implement.
We illustrate with the example \[
f\left(x\right)=\left\{ \begin{array}{ll}
-x & \mbox{ if }x<0\\
\sqrt{x} & \mbox{ if }x\geq0.\end{array}\right.\]
The robust regularization can be quickly calculated to be \[
\bar{f}_{\epsilon}\left(x\right)=\left\{ \begin{array}{ll}
\epsilon-x & \mbox{ if }x<\alpha\left(\epsilon\right)\\
\sqrt{\epsilon+x} & \mbox{ if }x\geq\alpha\left(\epsilon\right),\end{array}\right.\]
where $\alpha\left(\epsilon\right)=\frac{1+2\epsilon-\sqrt{1+8\epsilon}}{2}>-\epsilon$.
The minimizer of $f$ is $\alpha\left(0\right)$, and $f$ is not
Lipschitz there. To see this, observe that $\frac{f\left(\delta\right)-f\left(0\right)}{\delta-0}\rightarrow\infty$
as $\delta\rightarrow0$. But the robust regularization $\bar{f}_{\epsilon}$
is Lipschitz at its minimizer $\alpha\left(\epsilon\right)$; its
left and right derivatives there are $-1$ and $\frac{1}{2\sqrt{\epsilon+\alpha\left(\epsilon\right)}}$,
which are both finite. 

The sensitivity of $f$ at $0$ can be attributed to the lack of Lipschitz
continuity there. Lipschitz continuity is important in variational
analysis, and is well studied in the recent books \cite{RW98,Mor06}.
The existence of a finite Lipschitz constant on $f$ close to the
optimizer can be important in the problems from which the optimization
problem was derived.

There are two main aims in this paper. The first aim is to show that
robust regularization has a regularizing property: Even if the original
function $f$ is not Lipschitz at a point $x$, the robust regularization
can be Lipschitz there under various conditions. For example, in Corollary
\ref{cor:robust-regularization-lipschitz}, we prove that if the set
of points at which $f$ is not Lipschitz is isolated, then the robust
regularization $\bar{f}_{\epsilon}$ is Lipschitz at these points
for all small $\epsilon>0$. The second aim is to highlight the relationship
between calmness and Lipschitz continuity, a topic important in the
study of metric regularity, and studied in some generality for set-valued
mappings (for example, in \cite[Theorem 2.1]{Li94}, \cite[Theorem 1.5]{Rob07})
but exploited less for single-valued mappings.

In Theorem \ref{thm:lip<infty}, we prove that if $f:\mathbb{R}^{n}\rightarrow\mathbb{R}$
is semi-algebraic and continuous, then given any point in $\mathbb{R}^{n}$,
the robust regularization $\bar{f}_{\epsilon}$ is Lipschitz there
for all small $\epsilon>0$. Semi-algebraic functions are functions
whose graph can be defined by a finite union of sets defined by finitely
many polynomial equalities and inequalities, and is a broad class
of functions in applications. (For example, piecewise polynomial functions,
rational functions and the mapping from a matrix to its eigenvalues
are all semi-algebraic functions.) Moreover, the Lipschitz modulus
of $\bar{f}_{\epsilon}$ at $\bar{x}$ is of order $o\left(\frac{1}{\epsilon}\right)$.
This estimate of the Lipschitz modulus can be helpful for robust design.

Several interesting examples of robust regularization are tractable
to compute and optimize. For example, the robust regularization of
any strictly convex quadratic is a semidefinite -representable function,
tractable via semidefinite programming: see Section \ref{sec:quad-exa}.
The robust regularizations of the spectral abscissa and radius of
a nonsymmetric square matrix, which are the largest real part and
the largest norm respectively of the eigenvalues of a matrix, are
two more interesting examples. The robust regularization of the spectral
abscissa and spectral radius are also known as the pseudospectral
abscissa and the pseudospectral radius. The pseudospectral abscissa
is important in the study of the system $\frac{d}{dt}u\left(t\right)=Au\left(t\right)$,
and is easily calculated using the algorithm in \cite{BLO02Pseudo},
while the pseudospectral radius is important in the study of the system
$u_{t+1}=Au_{t}$, and is easily calculated using the algorithm in
\cite{MO05}. We refer the reader to \cite{Tre06} for more details
on the importance of the pseudospectral abscissa and radius in applications.
The spectral abscissa is nonlipschitz whenever the eigenvalue with
the largest real part has a nontrivial Jordan block. But for a fixed
matrix, the pseudospectral abscissa is Lipschitz there for all $\epsilon\in\left(0,\bar{\epsilon}\right)$
if $\bar{\epsilon}>0$ is small enough \cite{LP07}. We rederive this
result here, using a much more general approach.

\section{\label{sec:Calmness}Calmness as an extension to Lipschitzness}

We begin by discussing the relation between calmness and Lipschitz
continuity, which will be important in the proofs in Section \ref{sec:Semialgebraic}
later. Throughout the paper, we will limit ourselves to the single-valued
case. For more on these topics and their set-valued extensions, we
refer the reader to \cite{RW98}.

\begin{defn}
\label{def:calm-and-lip}Let $F:X\rightarrow\mathbb{R}^{m}$ be a
single-valued map, where $X\subset\mathbb{R}^{n}$.

(a) \cite[Section 8F]{RW98} Define the \emph{calmness modulus} of
$F$ at $\bar{x}$ with respect to $X$ to be \begin{eqnarray*}
\calmx F\left(\bar{x}\right) & := & \inf\{\kappa\mid\mbox{There is a neighbourhood }V\mbox{ of }\bar{x}\mbox{ such that }\\
 &  & \,\,\,\,\,\,\,\,\,\,\,\,\,\,\,\,\,\,\left|F\left(x\right)-F\left(\bar{x}\right)\right|\leq\kappa\left|x-\bar{x}\right|\mbox{ for all }x\in V\cap X\}\\
 & = & \limsup_{x\xrightarrow[X]{}\bar{x}}\frac{\left|F\left(x\right)-F\left(\bar{x}\right)\right|}{\left|x-\bar{x}\right|}.\end{eqnarray*}
Here, $x\xrightarrow[X]{}\bar{x}$ means that $x\in X$ and $x\rightarrow\bar{x}$.
The function $F$ is \emph{calm} at $\bar{x}$ with respect to $X$
if $\calmx F\left(\bar{x}\right)<\infty$. 

(b)\cite[Definition 9.1]{RW98} Define the \emph{Lipschitz modulus}
of $F$ at $\bar{x}$ with respect to $X$ to be \begin{eqnarray*}
\lipx F\left(\bar{x}\right) & := & \inf\{\kappa\mid\mbox{There is a neighbourhood }V\mbox{ of }\bar{x}\mbox{ such that }\\
 &  & \,\,\,\,\,\,\,\,\,\,\,\,\,\,\,\,\,\,\left|F\left(x\right)-F\left(x^{\prime}\right)\right|\leq\kappa\left|x-x^{\prime}\right|\mbox{ for all }x,x^{\prime}\in V\cap X\}\\
 & = & \limsup_{{x,x^{\prime}\xrightarrow[X]{}\bar{x}\atop x\neq x^{\prime}}}\frac{\left|F\left(x\right)-F\left(x^{\prime}\right)\right|}{\left|x-x^{\prime}\right|}.\end{eqnarray*}
The function $F$ is \emph{Lipschitz} at $\bar{x}$ with respect to
$X$ if $\lipx F\left(\bar{x}\right)<\infty$. $\diamond$
\end{defn}
The definitions differ slightly from that of \cite{RW98}. As can
be seen in the definitions, Lipschitz continuity is a more stringent
form of continuity than calmness. In fact, they are related in the
following manner. 

\begin{prop}
\label{pro:lip-eq-cl-calm}Suppose that $F:X\rightarrow\mathbb{R}^{m}$
where $X\subset\mathbb{R}^{n}$. 

(a) $\limsup_{x\xrightarrow[X]{}\bar{x}}\calmx F\left(x\right)\leq\lipx F\left(\bar{x}\right)$.

(b) If there is an open set $U$ containing $\bar{x}$ such that $U\cap X$
is convex, then $\lipx F\left(\bar{x}\right)=\limsup_{x\xrightarrow[X]{}\bar{x}}\calmx F\left(x\right)$.

\end{prop}
\begin{proof}
To simplify notation, let $\kappa:=\limsup_{x\xrightarrow[X]{}\bar{x}}\calmx F\left(x\right)$. 

(a) For any $\epsilon>0$, we can find a point $x_{\epsilon}$ such
that $\left|\bar{x}-x_{\epsilon}\right|<\epsilon$ and $\calmx F\left(x_{\epsilon}\right)>\kappa-\epsilon$.
Then we can find a point $\tilde{x}_{\epsilon}$ such that $\left|x_{\epsilon}-\tilde{x}_{\epsilon}\right|<\epsilon$
and $\left|F\left(x_{\epsilon}\right)-F\left(\tilde{x}_{\epsilon}\right)\right|>\left(\kappa-\epsilon\right)\left|x_{\epsilon}-\tilde{x}_{\epsilon}\right|$.
As $\epsilon$ can be made arbitrarily small, we have $\kappa\leq\lipx F\left(\bar{x}\right)$
as needed. 

(b) For every $\epsilon>0$ , there is some neighborhood of $\bar{x}$,
say $\mathbb{B}_{\delta}\left(\bar{x}\right)$, such that \[
\calmx F\left(x\right)\leq\kappa+\epsilon\mbox{ if }x\in\mathbb{B}_{\delta}\left(\bar{x}\right)\cap X.\]
For any $y,z\in\mathbb{B}_{\delta}\left(\bar{x}\right)\cap X$, consider
the line segment joining $y$ and $z$, which we denote $\left[y,z\right]$.
As $\calmx F\left(\tilde{x}\right)\leq\kappa+\epsilon$ for all $\tilde{x}\in\left[y,z\right]$,
there is a neighborhood around $\tilde{x}$, say $V_{\tilde{x}}$,
such that $\left|F\left(\hat{x}\right)-F\left(\tilde{x}\right)\right|\leq\left(\kappa+2\epsilon\right)\left|\hat{x}-\tilde{x}\right|$
for all $\hat{x}\in V_{\tilde{x}}\cap X$. 

As $\left[y,z\right]$ is compact, choose finitely many $\tilde{x}$
such that the union of $V_{\tilde{x}}$ covers $\left[y,z\right]$.
We can add $y$ and $z$ into our choice of points and rename them
as $\tilde{x}_{1},\dots,\tilde{x}_{k}$ in their order on the line
segment $\left[y,z\right]$, with $\tilde{x}_{1}=y$ and $\tilde{x}_{k}=z$.
Also, we can find a point $\hat{x}_{i}$ between $\tilde{x}_{i}$
and $\tilde{x}_{i+1}$ such that $\hat{x}_{i}\in V_{\tilde{x}_{i}}\cap V_{\tilde{x}_{i+1}}$.
Therefore, we add these $\hat{x}_{i}$ into $\tilde{x}_{1},\dots,\tilde{x}_{k}$
and get a new set $x_{1},\dots,x_{K}$, again in their order on the
line segment and $x_{1}=y$, $x_{K}=z$. 

We have:\begin{eqnarray*}
\left|F\left(y\right)-F\left(z\right)\right| & \leq & \sum_{i=1}^{K-1}\left|F\left(x_{i}\right)-F\left(x_{i+1}\right)\right|\\
 & \leq & \sum_{i=1}^{K-1}\left(\kappa+2\epsilon\right)\left|x_{i}-x_{i+1}\right|\\
 & \leq & \left(\kappa+2\epsilon\right)\left|y-z\right|,\end{eqnarray*}
and as $\epsilon$ is arbitrary, $\lipx F\left(\bar{x}\right)\leq\kappa$
as claimed.
\end{proof}
Convexity is a strong assumption here, but some analogous condition
is needed, as the following examples show.

\begin{example}
(a) Consider the set $X\subset\mathbb{R}$ defined by \[
X=\left(\bigcup_{i=1}^{\infty}\left[\frac{1}{3^{i}},\frac{2}{3^{i}}\right]\right)\cup\left\{ 0\right\} ,\]
 and define the function $F:X\rightarrow\mathbb{R}$ by \[
F\left(x\right)=\left\{ \begin{array}{ll}
\frac{1}{3^{i}} & \mbox{if }\frac{1}{3^{i}}\leq x\leq\frac{2}{3^{i}},\\
0 & \mbox{if }x=0.\end{array}\right.\]
It is clear that $\calmx F\left(x\right)=0$ for all $x\in X\backslash\left\{ 0\right\} $
since $F$ is constant on each component of $X$, and $\calmx F\left(0\right)=1$.
But \begin{eqnarray*}
\lipx F\left(0\right) & = & \lim_{i\rightarrow\infty}\frac{F\left(\frac{1}{3^{i}}\right)-F\left(\frac{2}{3^{i+1}}\right)}{\frac{1}{3^{i}}-\frac{2}{3^{i+1}}}\\
 & = & \lim_{i\rightarrow\infty}\frac{\frac{1}{3^{i}}-\frac{1}{3^{i+1}}}{\frac{1}{3^{i}}-\frac{2}{3^{i+1}}}\\
 & = & 2.\end{eqnarray*}
Thus, $\limsup_{x\rightarrow0}\calmx F\left(x\right)<\lipx F\left(0\right)$.

(b) Consider $X\subset\mathbb{R}^{2}$ defined by $X:=\left\{ \left(x_{1},x_{2}\right)\mid x_{2}^{2}=x_{1}^{4}\right\} $
and the function $F:\mathbb{R}^{2}\rightarrow\mathbb{R}$ defined
by $F\left(x_{1},x_{2}\right)=x_{2}$. One can easily check that $\limsup_{x\rightarrow0}\calmx F\left(x\right)=0$
and $\lipx F\left(0,0\right)=1$. This is an example of a semi-algebraic
function where inequality holds.$\diamond$
\end{example}
Note that $\calm F\left(\bar{x}\right)$ can be strictly smaller than
$\lip F\left(\bar{x}\right)$ even if $X$ is convex, as demonstrated
below.

\begin{example}
(a) Consider $F:\mathbb{R}\rightarrow\mathbb{R}$ defined by \[
F\left(x\right)=\left\{ \begin{array}{ll}
0 & \mbox{if }x=0,\\
x^{2}\sin\left(\frac{1}{x^{2}}\right) & \mbox{otherwise.}\end{array}\right.\]
Here, $\calm F\left(0\right)=0$, but $\lip F\left(0\right)=\infty$.

(b) \label{exa:calm-neq-lip}Consider $F:\mathbb{R}^{2}\rightarrow\mathbb{R}$
defined by: \[
F\left(x_{1},x_{2}\right)=\left\{ \begin{array}{ll}
0 & \mbox{if }x_{1}\leq0\\
x_{1} & \mbox{if }0\leq x_{1}\leq x_{2}/2\\
-x_{1} & \mbox{if }0\leq x_{1}\leq-x_{2}/2\\
2x_{2} & \mbox{if }x_{1}\geq\left|x_{2}\right|/2.\end{array}\right.\]
We can calculate $\calm F\left(0,0\right)=2/\sqrt{5}$, and $\lip F\left(0,0\right)=2$,
so this gives $\calm F\left(0,0\right)<\lip F\left(0,0\right)$. This
is an example of a semi-algebraic function where inequality holds.$\diamond$
\end{example}
At this point, we make a remark about subdifferentially regular functions.
We recall the definition of subdifferential regularity.

\begin{defn}
\cite[Definition 8.3]{RW98} \label{def:subdifferential}Consider
a function $f:\mathbb{R}^{n}\rightarrow\mathbb{R}\cup\left\{ \infty\right\} $
and a point $\bar{x}$ with $f\left(\bar{x}\right)$ finite. For a
vector $v\in\mathbb{R}^{n}$, one says that 

(a) $v$ is a \emph{regular subgradient} of $f$ at $\bar{x}$, written
$v\in\hat{\partial}f\left(\bar{x}\right)$, if \[
f\left(x\right)\geq f\left(\bar{x}\right)+\left\langle v,x-\bar{x}\right\rangle +o\left(\left|x-\bar{x}\right|\right);\]

(b) $v$ is a \emph{(general) subgradient} of $f$ at $\bar{x}$,
written $v\in\partial f\left(\bar{x}\right)$, if there are sequences
$x^{\nu}\rightarrow\bar{x}$ and $v^{\nu}\in\hat{\partial}f\left(x^{\nu}\right)$
with $v^{\nu}\rightarrow v$ and $f\left(x^{\nu}\right)\rightarrow f\left(\bar{x}\right)$.

(c) If $f$ is Lipschitz continuous at $\bar{x}$, then $f$ is \emph{subdifferentially
regular} if $\hat{\partial}f\left(\bar{x}\right)=\partial f\left(\bar{x}\right)$. 
\end{defn}
Though the definition of subdifferential regularity differs from that
given in \cite[Definition 7.25]{RW98}, it can be deduced from \cite[Corollary 8.11, Theorem 9.13 and Theorem 8.6]{RW98}
when $f$ is Lipschitz, and is simple enough for our purposes. Subdifferentially
regular functions are important and well-studied in variational analysis.
The class of subdifferentially regular functions is closed under sums
and pointwise maxima, and includes smooth functions and convex functions.
It turns out that the calmness and Lipschitz moduli are equal for
subdifferentially regular functions.

\begin{prop}
If $f:\mathbb{R}^{n}\rightarrow\mathbb{R}\cup\left\{ \infty\right\} $
is Lipschitz continuous at $\bar{x}$ and subdifferentially regular
there, then $\calm f\left(\bar{x}\right)=\lip f\left(\bar{x}\right)$.
\end{prop}
\begin{proof}
By \cite[Theorem 9.13]{RW98}, $\lip f\left(\bar{x}\right)=\max\left\{ \left|v\right|\mid v\in\partial f\left(\bar{x}\right)\right\} $.
If $v\in\partial f\left(\bar{x}\right)$, then $v\in\hat{\partial}f\left(\bar{x}\right)$,
and we observe that $\calm f\left(\bar{x}\right)\geq\left|v\right|$
because \begin{eqnarray*}
f\left(\bar{x}+tv\right) & \geq & f\left(\bar{x}\right)+\left\langle v,tv\right\rangle +o\left(\left|t\right|\right)\\
 & = & f\left(\bar{x}\right)+\left|v\right|\left|tv\right|+o\left(\left|t\right|\right).\end{eqnarray*}
Therefore $\calm f\left(\bar{x}\right)\leq\lip f\left(\bar{x}\right)=\max\left\{ \left|v\right|\mid v\in\partial f\left(\bar{x}\right)\right\} \leq\calm f\left(\bar{x}\right)$,
which implies that all three terms are equal.
\end{proof}

\section{Calmness and robust regularization}

Recall the definition of robust regularization in Definition \ref{def:robust-regularization}.
To study robust regularization, it is useful to study the dependence
of $\bar{f}_{\epsilon}\left(x\right)$ on $\epsilon$ instead of on
$x$. For a point $x\in X$, define $g_{x}:\mathbb{R}_{+}\rightarrow\mathbb{R}$
by \[
g_{x}\left(\epsilon\right)=\bar{f}_{\epsilon}\left(x\right).\]
To simplify notation, we write $g\equiv g_{x}$ if it is clear from
context. Here are a few basic properties of $g_{x}$.

\begin{prop}
\label{pro:ppties_of_g}For $f:X\rightarrow\mathbb{R}$ and $g_{x}$
as defined above, we have the following:

(a) $g_{x}$ is monotonically nondecreasing. 

(b) If $f$ is continuous in a neighborhood of $x$, then $g_{x}$
is continuous in a neighborhood of $0$.
\end{prop}
\begin{proof}
Part (a) is obvious. For part (b), we prove the left and right limits
separately. Suppose that $\epsilon_{i}\downarrow\epsilon$. There
is a sequence of $x_{i}$ such that $f\left(x_{i}\right)=\bar{f}_{\epsilon_{i}}\left(x\right)$,
and $\left|x_{i}-x\right|\leq\epsilon_{i}$. We assume, by choosing
a subsequence if needed, that $\lim_{i\rightarrow\infty}x_{i}=\tilde{x}$.
We have $\left|\tilde{x}-x\right|\leq\epsilon$, and since $f$ is
continuous, $f\left(x_{i}\right)\rightarrow f\left(\tilde{x}\right)$.
This means that \[
\bar{f}_{\epsilon}\left(x\right)\geq f\left(\tilde{x}\right)=\lim_{i\rightarrow\infty}\bar{f}_{\epsilon_{i}}\left(x\right),\]
which implies $g\left(\epsilon\right)\geq\limsup_{\tilde{\epsilon}\downarrow\epsilon}g\left(\tilde{\epsilon}\right)$.
The monotonicity of $g$ tells us that $g\left(\epsilon\right)=\lim_{\tilde{\epsilon}\downarrow\epsilon}g\left(\tilde{\epsilon}\right)$.

Next, suppose that $\epsilon_{i}$ increases monotonically to $\epsilon$.
Let $\hat{x}$ be such that $f\left(\hat{x}\right)=\bar{f}_{\epsilon}\left(x\right)$,
with $\left|\hat{x}-x\right|\leq\epsilon$. Since $f$ is continuous,
for every $\delta_{1}>0$, there is a $\delta_{2}>0$ such that $\left|f\left(x^{\prime}\right)-f\left(\hat{x}\right)\right|<\delta_{1}$
if $\left|x^{\prime}-\hat{x}\right|<\delta_{2}$. This means that
if $\epsilon-\epsilon_{i}<\delta_{2}$, then \[
\bar{f}_{\epsilon_{i}}\left(x\right)\geq f\left(\hat{x}\right)-\delta_{1}=\bar{f}_{\epsilon}\left(x\right)-\delta_{1}.\]
 As $\delta_{1}$ can be made arbitrarily small, we conclude that
$\lim_{\tilde{\epsilon}\uparrow\epsilon}\bar{f}_{\tilde{\epsilon}}\left(x\right)=\bar{f}_{\epsilon}\left(x\right)$,
or $\lim_{\tilde{\epsilon}\uparrow\epsilon}g\left(\tilde{\epsilon}\right)=g\left(\epsilon\right)$.
\end{proof}
It turns out that calmness of the robust regularization is related
to the derivative of $g_{x}$.

\begin{prop}
\label{pro:calm-eq-derivative} If $f:X\rightarrow\mathbb{R}$ and
$\epsilon>0$, then $\calmx\bar{f}_{\epsilon}\left(x\right)\leq\calm g_{x}\left(\epsilon\right)$.
If in addition $X=\mathbb{R}^{n}$ and $g_{x}$ is differentiable
at $\epsilon$, then \[
\calm\bar{f}_{\epsilon}\left(x\right)=\calm g_{x}\left(\epsilon\right)=g_{x}^{\prime}\left(\epsilon\right).\]

\end{prop}
\begin{proof}
For the first part, we proceed to show that if $\kappa>\calm g_{x}\left(\epsilon\right)$,
then $\kappa\geq\calm\bar{f}_{\epsilon}\left(x\right)$. If $\left|\tilde{x}-x\right|<\epsilon$,
we have \begin{eqnarray*}
\mathbb{B}_{\epsilon-\left|\tilde{x}-x\right|}\left(x\right) & \subset\mathbb{B}_{\epsilon}\left(\tilde{x}\right)\subset & \mathbb{B}_{\epsilon+\left|\tilde{x}-x\right|}\left(x\right),\end{eqnarray*}
which implies\begin{eqnarray*}
\bar{f}_{\epsilon-\left|\tilde{x}-x\right|}\left(x\right) & \leq\bar{f}_{\epsilon}\left(\tilde{x}\right)\leq & \bar{f}_{\epsilon+\left|\tilde{x}-x\right|}\left(x\right).\end{eqnarray*}
Then note that if $\tilde{x}$ is close enough to $x$, we have

\[
\bar{f}_{\epsilon}\left(\tilde{x}\right)\leq\bar{f}_{\epsilon+\left|\tilde{x}-x\right|}\left(x\right)=g_{x}\left(\epsilon+\left|\tilde{x}-x\right|\right)\leq g_{x}\left(\epsilon\right)+\kappa\left|\tilde{x}-x\right|,\]
and similarly \[
\bar{f}_{\epsilon}\left(\tilde{x}\right)\geq\bar{f}_{\epsilon-\left|\tilde{x}-x\right|}\left(x\right)=g_{x}\left(\epsilon-\left|\tilde{x}-x\right|\right)\geq g_{x}\left(\epsilon\right)-\kappa\left|\tilde{x}-x\right|,\]
which tells us that $\left|\bar{f}_{\epsilon}\left(\tilde{x}\right)-\bar{f}_{\epsilon}\left(x\right)\right|\leq\kappa\left|\tilde{x}-x\right|$,
which is what we need. 

For the second part, it is clear from the definition of the derivative
that $g_{x}^{\prime}\left(\epsilon\right)=\calm g_{x}\left(\epsilon\right)$.
We prove that if $\kappa<g_{x}^{\prime}\left(\epsilon\right)$, then
$\kappa\leq\calm\bar{f}_{\epsilon}\left(x\right)$. By the differentiability
of $g_{x}$, there is some $\bar{\delta}>0$ such that for any $0\leq\delta\leq\bar{\delta}$,
we have \begin{eqnarray*}
\bar{f}_{\epsilon+\delta}\left(x\right) & = & g_{x}\left(\epsilon+\delta\right)\\
 & > & g_{x}\left(\epsilon\right)+\kappa\delta\\
 & = & \bar{f}_{\epsilon}\left(x\right)+\kappa\delta.\end{eqnarray*}
For any $0\leq\delta\leq\bar{\delta}$, there is some $\tilde{x}_{\delta}\in\mathbb{B}_{\epsilon+\delta}\left(x\right)$
such that $f\left(\tilde{x}_{\delta}\right)=\bar{f}_{\epsilon+\delta}\left(x\right)$.
Let $\hat{x}_{\delta}=\frac{\delta}{\left|\tilde{x}_{\delta}-x\right|}\left(\tilde{x}_{\delta}-x\right)+x$.
We have $\bar{f}_{\epsilon}\left(\hat{x}_{\delta}\right)=\bar{f}_{\epsilon+\delta}\left(x\right)$,
which gives $\bar{f}_{\epsilon}\left(\hat{x}_{\delta}\right)-\bar{f}_{\epsilon}\left(x\right)>\kappa\delta$.
Since $\hat{x}_{\delta}$ was chosen such that $\delta=\left|\hat{x}_{\delta}-x\right|$,
we have $\bar{f}_{\epsilon}\left(\hat{x}_{\delta}\right)-\bar{f}_{\epsilon}\left(x\right)>\kappa\left|\hat{x}_{\delta}-x\right|$,
which implies $\kappa\leq\calm\bar{f}_{\epsilon}\left(x\right)$ as
needed.
\end{proof}
\begin{rem}
A similar statement can be made for $\epsilon=0$, except that we
change calmness to {}``calm from above'' as defined in \cite[Section 8F]{RW98}
in both parts.
\end{rem}
We have the following corollary. The subdifferential {}``$\partial$''
was defined in Definition \ref{def:subdifferential}.

\begin{cor}
\label{cor:calm-leq-lip}If $f:\mathbb{R}^{n}\rightarrow\mathbb{R}$,
$\epsilon>0$ and $g_{x}$ is Lipschitz at $\epsilon$, then \[
\calm\bar{f}_{\epsilon}\left(x\right)\leq\lip g_{x}\left(\epsilon\right)=\sup\left\{ \left|y\right|\mid y\in\partial g_{x}\left(\epsilon\right)\right\} .\]

\end{cor}
\begin{proof}
It is clear that $\calm\bar{f}_{\epsilon}\left(x\right)\leq\calm g_{x}\left(\epsilon\right)\leq\lip g_{x}\left(\epsilon\right)$.
The formula $\lip g_{x}\left(\epsilon\right)$$=\sup\{\left|y\right|\mid y\in\partial g_{x}\left(\epsilon\right)\}$
follows from \cite[Theorem 9.13, Definition 9.1]{RW98}.
\end{proof}
In general, the robust regularization is calm.

\begin{prop}
\label{pro:calmness-a.e.}For a continuous function $f:X\rightarrow\mathbb{R}$,
there is an $\bar{\epsilon}>0$ such that $\bar{f}_{\epsilon}$ is
calm at $x$ for all $0<\epsilon\leq\bar{\epsilon}$ except on a subset
of $(0,\bar{\epsilon}]$ of measure zero. 
\end{prop}
\begin{proof}
By Proposition \ref{pro:ppties_of_g}(b), since $f$ is continuous
at $x$, $g_{x}$ is continuous in $\left[0,\bar{\epsilon}\right]$
for some $\bar{\epsilon}>0$. Since $g_{x}$ is monotonically nondecreasing,
it is differentiable in all $\left[0,\bar{\epsilon}\right]$ except
for a set of measure zero. The derivative $g_{x}^{\prime}\left(\epsilon\right)$
equals $\mbox{calm }\bar{f}_{\epsilon}\left(x\right)$ by Proposition
\ref{pro:calm-eq-derivative}. 
\end{proof}
\begin{rem}
In general, the above result cannot be improved. For an example, let
$c:\left[0,1\right]\rightarrow\left[0,1\right]$ denote the Cantor
function, commonly used in real analysis texts as an example of a
function that is not absolutely continuous and not satisfying the
Fundamental Theorem of Calculus. Then $\calm\bar{c}_{\epsilon}\left(0\right)=\infty$
for all $\epsilon$ lying in the Cantor set.$\diamond$
\end{rem}

\section{\label{sec:general-r-r}Robust regularization in general}

In this section, in Corollary \ref{cor:robust-regularization-lipschitz},
we prove that if $\lip f\left(x\right)<\infty$ for $x$ close to
but not equal to $\bar{x}$, then $\lip\bar{f}_{\epsilon}\left(\bar{x}\right)<\infty$
for all small $\epsilon>0$, even when $\lip f\left(\bar{x}\right)=\infty$.
To present the details of the proof, we need a short foray into set-valued
analysis.

\begin{defn}
\cite[Example 4.13]{RW98} For two sets $C,D\subset\mathbb{R}^{m}$,
the \emph{Pompieu-Hausdorff distance} between $C$ and $D$, denoted
by $\mathbf{d}\left(C,D\right)$, is defined by \[
\mathbf{d}\left(C,D\right):=\inf\left\{ \eta\geq0\mid C\subset D+\eta\mathbb{B},D\subset C+\eta\mathbb{B}\right\} .\]

\begin{defn}
\cite[Definitions 9.26, 9.28]{RW98} A mapping $S:X\rightrightarrows\mathbb{R}^{m}$
is \emph{Lipschitz continuous} on its domain $X\subset\mathbb{R}^{n}$,
if it is nonempty-closed-valued on $X$ and there exists $\kappa\ge0$,
a Lipschitz constant, such that \[
\mathbf{d}\left(S\left(x^{\prime}\right),S\left(x\right)\right)\leq\kappa\left|x^{\prime}-x\right|\mbox{ for all }x,x^{\prime}\in X,\]
 or equivalently, $S\left(x^{\prime}\right)\subset S\left(x\right)+\kappa\left|x^{\prime}-x\right|\mathbb{B}$
for all $x,x^{\prime}\in X$. The \emph{Lipschitz modulus} is defined
as \[
\lip S\left(\bar{x}\right):=\limsup_{{x,x^{\prime}\xrightarrow[X]{}\bar{x}\atop x\neq x^{\prime}}}\frac{\mathbf{d}\left(S\left(x^{\prime}\right),S\left(x\right)\right)}{\left|x^{\prime}-x\right|},\]
and is the infimum of all $\kappa$ such that there exists a neighborhood
$U$ of $\bar{x}$ such that $S$ is Lipschitz continuous with constant
$\kappa$ in $U\cap X$.$\diamond$
\end{defn}
\end{defn}
For $F:X\rightarrow\mathbb{R}^{m}$, we may write the robust regularization
$F_{\epsilon}:X\rightrightarrows\mathbb{R}^{m}$ by $F_{\epsilon}=F\circ\Phi_{\epsilon}$,
where $\Phi_{\epsilon}:X\rightrightarrows X$ is defined by $\Phi_{\epsilon}\left(x\right)=\mathbb{B}_{\epsilon}\left(x\right)\cap X$.
For reasons that will be clear later in Section \ref{sec:The-condition},
we consider the extension $\tilde{\Phi}_{\epsilon}:\mathbb{R}^{n}\rightrightarrows X$
defined by $\tilde{\Phi}_{\epsilon}\left(x\right)=\mathbb{B}_{\epsilon}\left(x\right)\cap X$.
It is clear that $\tilde{\Phi}_{\epsilon}\mid_{X}=\Phi_{\epsilon}$
using our previous notation, and it follows straight from the definitions
that $\lip\Phi_{\epsilon}\left(x\right)\leq\lip\tilde{\Phi}_{\epsilon}\left(x\right)$
for $x\in X$. 

\begin{defn}
We say that $X\subset\mathbb{R}^{n}$ is \emph{peaceful at} $\bar{x}\in X$
if $\lip\Phi_{\epsilon}\left(\bar{x}\right)$ is finite for all small
$\epsilon>0$. If in addition $\limsup_{\epsilon\downarrow0}\lip\tilde{\Phi}_{\epsilon}\left(\bar{x}\right)\leq\kappa$
for all small $\epsilon>0$, we say that $X$ is peaceful with modulus
$\kappa$ at $\bar{x}$, or \emph{$\kappa$-peaceful} \emph{at} $\bar{x}$.
\end{defn}
When $\bar{x}$ lies in the interior of $X$ and $\epsilon$ is small
enough, then $\tilde{\Phi}_{\epsilon}$ is Lipschitz with constant
$1$. In section \ref{sec:The-condition}, we will find weaker conditions
on $X$ for the Lipschitz continuity of $\tilde{\Phi}_{\epsilon}$.
We will see that convex sets are $1$-peaceful, but for now, we remark
that if $X$ is convex, then $\Phi_{\epsilon}$ is globally Lipschitz
in $X$.

\begin{prop}
\label{pro:convex-lip-one}If $X$ is a convex set, then $\Phi_{\epsilon}\left(x\right)\subset\Phi\left(x^{\prime}\right)+\left|x-x^{\prime}\right|\mathbb{B}$
for all $x,x^{\prime}\in X$.
\end{prop}
\begin{proof}
The condition we are required to prove is equivalent to \[
\mathbb{B}_{\epsilon}\left(x\right)\cap X\subset\left(\mathbb{B}_{\epsilon}\left(x^{\prime}\right)\cap X\right)+\left|x-x^{\prime}\right|\mathbb{B}\mbox{ for }x,x^{\prime}\in X.\]
For any point $\tilde{x}\in\mathbb{B}_{\epsilon}\left(x\right)\cap X$,
the line segment $\left[x^{\prime},\tilde{x}\right]$ lies in $X$,
and is of length at most $\left|\tilde{x}-x\right|+\left|x-x^{\prime}\right|$.
The ball $\mathbb{B}_{\epsilon}\left(x^{\prime}\right)$ can contain
the line segment $\left[x^{\prime},\tilde{x}\right]$, in which case
$\tilde{x}\in\mathbb{B}_{\epsilon}\left(x^{\prime}\right)\cap X$,
or the boundary of $\mathbb{B}_{\epsilon}\left(x^{\prime}\right)$
may intersect $\left[x^{\prime},\tilde{x}\right]$ at a point, say
$\hat{x}$. Since $X$ is a convex set, we have $\hat{x}\in\mathbb{B}_{\epsilon}\left(x^{\prime}\right)\cap X$.
Furthermore \begin{eqnarray*}
\left|\tilde{x}-\hat{x}\right| & = & \left|\tilde{x}-x^{\prime}\right|-\epsilon\\
 & \leq & \left|\tilde{x}-x\right|+\left|x-x^{\prime}\right|-\epsilon\\
 & \leq & \left|x-x^{\prime}\right|,\end{eqnarray*}
so $\tilde{x}\in\left(\mathbb{B}_{\epsilon}\left(x^{\prime}\right)\cap X\right)+\left|x-x^{\prime}\right|\mathbb{B}$.
\end{proof}
We remark that if $X$ is nearly radial at $\bar{x}$ as introduced
in \cite{L02}, then $X$ is $1$-peaceful: see Section \ref{sec:The-condition}.
The set $X$ is \emph{nearly radial} \emph{at $\bar{x}$} if \[
\mbox{dist}\left(\bar{x},x+T_{X}\left(x\right)\right)\rightarrow0\mbox{ as }x\rightarrow\bar{x}\mbox{ in }X.\]
The set $X$ is \emph{nearly radial} if it is nearly radial at all
points in $X$. The notation $T_{X}\left(x\right)$ refers to the
\emph{(Bouligand) tangent cone} (or {}``contingent cone'') to $X$
at $x\in X$, formally defined as \[
T_{X}(\bar{x})=\{\lim t_{r}^{-1}(x_{r}-\bar{x}):t_{r}\downarrow0,\,\,\, x_{r}\rightarrow\bar{x},\,\,\, x_{r}\in X\}\]
 (see, for example, \cite[Definition 6.1]{RW98}). Many sets are nearly
radial, including for instance semi-algebraic sets, amenable sets
and smooth manifolds. 

We now present a result on the regularizing property of robust regularization.
In Proposition \ref{pro:inf-lip-regularization} below, condition
(i) allows us to evaluate the Lipschitz modulus of functions whose
domains are not necessarily convex. One situation where (i) is interesting
is when $X$ is a smooth manifold.

\begin{prop}
\label{pro:inf-lip-regularization}For $F:X\rightarrow\mathbb{R}^{m}$,
suppose that either (i) or (ii) holds.

(i) $X$ is peaceful and $\lip\tilde{F}\left(x\right)<\infty$ for
all $x$ close to but not equal to $\bar{x}$. Here, $\tilde{F}:\mathbb{R}^{n}\rightarrow\mathbb{R}^{m}$
is an extension of $F$ on $\mathbb{R}^{n}$ such that $\tilde{F}|_{X}=F$.

(ii) $X$ is convex and $\lipx F\left(x\right)<\infty$ for all $x\in X$
close to but not equal to $\bar{x}$. 

Then $\lip F_{\epsilon}\left(\bar{x}\right)$ is finite for all small
$\epsilon>0$. 
\end{prop}
\begin{proof}
The proof for both conditions are similar, so they will be treated
together. One notes that $\lipx F\left(x\right)\leq\lip\tilde{F}\left(x\right)$
always by the definition of these Lipschitz moduli, so we assume $\lipx F\left(x\right)<\infty$
for all $x\in X$ close to but not equal to $\bar{x}$ until we have
to distinguish these cases.

First, we prove that $\lipx F:X\rightarrow\mathbb{R}_{+}$ is upper
semicontinuous. This result is just a slight modification of the first
part of \cite[Theorem 9.2]{RW98}, but we include the proof for completeness.
Suppose that $x_{i}\rightarrow x$. By the definition of $\lipx F$,
we can find $x_{i,1}$, $x_{i,2}\in X$ such that \begin{eqnarray*}
\frac{\left|F\left(x_{i,1}\right)-F\left(x_{i,2}\right)\right|}{\left|x_{i,1}-x_{i,2}\right|} & > & \lipx F\left(x_{i}\right)-\left|x_{i}-x\right|,\\
\mbox{ and }\left|x_{i,j}-x_{i}\right| & < & \left|x_{i}-x\right|\mbox{ for }j=1,2.\end{eqnarray*}
Taking limits as $i\rightarrow\infty$, we see that $x_{i,1},x_{i,2}\rightarrow x$,
and it follows that \begin{eqnarray*}
\lipx F\left(x\right) & \geq & \limsup_{i\rightarrow\infty}\frac{\left|F\left(x_{i,1}\right)-F\left(x_{i,2}\right)\right|}{\left|x_{i,1}-x_{i,2}\right|}\\
 & = & \limsup_{i\rightarrow\infty}\lipx F\left(x_{i}\right).\end{eqnarray*}
Thus $\lipx F:X\rightarrow\mathbb{R}_{+}$ is upper semicontinuous.

So for $\epsilon_{1}$ small enough, choose $\epsilon_{2}<\epsilon_{1}$
such that $\lipx F$ is bounded above in $C_{1}=\left(\mathbb{B}_{\epsilon_{1}+\epsilon_{2}}\left(\bar{x}\right)\backslash\mathbb{B}_{\epsilon_{1}-\epsilon_{2}}\left(\bar{x}\right)\right)\cap X$,
say by the constant $\kappa_{1}$. Then for any $\kappa_{2}>\kappa_{1}$
and any $x\in C_{1}$, there is an $\epsilon_{x}$ such that $F$
is Lipschitz continuous on $\mathbb{B}_{\epsilon_{x}}\left(x\right)\cap X$
with constant $\kappa_{2}$ with respect to $X$. Thus $\cup_{x\in C_{1}}\left\{ \mathbb{B}_{\epsilon_{x}}\left(x\right)\right\} $
is an open cover of $C_{1}$.

By the Lebesgue Number Lemma, there is a constant $\delta$ such that
if $x_{1},x_{2}$ lie in $C_{1}$ and $\left|x_{1}-x_{2}\right|\leq\delta$,
then the line segment $\left[x_{1},x_{2}\right]$ lies in one of the
open balls $\mathbb{B}_{\epsilon_{x}}\left(x\right)$ for some $x\in C_{1}$.
We may assume that $\delta<\epsilon_{2}$.

Also, since $X$ is peaceful at $\bar{x}$, choose $\epsilon_{1}$
small enough so that $\lipi\Phi_{\epsilon_{1}}\left(\bar{x}\right)$
is finite, say $\lipi\Phi_{\epsilon_{1}}\left(\bar{x}\right)<K$.
If $X$ is convex, then this is possible due to Proposition \ref{pro:convex-lip-one}.
We can assume that $K>2$. Therefore, there is an open set $U$ about
$\bar{x}$ such that $\Phi_{\epsilon_{1}}$ is Lipschitz in $U\cap X$
with constant $K$, that is $\Phi_{\epsilon_{1}}\left(x\right)\subset\Phi_{\epsilon_{1}}\left(x^{\prime}\right)+K\left|x-x^{\prime}\right|\mathbb{B}$
for all $x,x^{\prime}\in U\cap X$.

So, for $x,x^{\prime}\in U\cap\mathbb{B}_{\frac{\delta}{2K}}\left(\bar{x}\right)\cap X$,
we want to show that \[
F_{\epsilon_{1}}\left(x\right)\subset F_{\epsilon_{1}}\left(x^{\prime}\right)+K\kappa_{2}\left|x-x^{\prime}\right|\mathbb{B}.\]
Suppose that $y\in F_{\epsilon_{1}}\left(x\right)$. So $y=F\left(\tilde{x}\right)$
for some $\tilde{x}\in\mathbb{B}_{\epsilon_{1}}\left(x\right)\cap X$.
If $\tilde{x}\in\mathbb{B}_{\epsilon_{1}-\frac{\delta}{2K}}\left(\bar{x}\right)$,
then $\tilde{x}\in\mathbb{B}_{\epsilon_{1}}\left(x^{\prime}\right)\cap X$
because $\left|x^{\prime}-\bar{x}\right|\leq\frac{\delta}{2K}$. So
$y\in F_{\epsilon_{1}}\left(x^{\prime}\right)$. Otherwise $\tilde{x}\in\left(\mathbb{B}_{\epsilon_{1}+\frac{\delta}{2K}}\left(\bar{x}\right)\backslash\mathbb{B}_{\epsilon_{1}-\frac{\delta}{2K}}\left(\bar{x}\right)\right)\cap X.$

We have $\Phi_{\epsilon_{1}}\left(x\right)\subset\Phi_{\epsilon_{1}}\left(x^{\prime}\right)+K\left|x-x^{\prime}\right|\mathbb{B}$.
So there is some $\hat{x}\in\Phi_{\epsilon_{1}}\left(x^{\prime}\right)$
such that \[
\left|\hat{x}-\tilde{x}\right|\leq K\left|x-x^{\prime}\right|\leq K\frac{\delta}{2K}=\frac{\delta}{2}.\]
Furthermore, \[
\left|\hat{x}-\bar{x}\right|\leq\left|\tilde{x}-x\right|+\left|x-\bar{x}\right|+\left|\hat{x}-\tilde{x}\right|\leq\epsilon_{1}+\frac{\delta}{2K}+\frac{\delta}{2}\leq\epsilon_{1}+\frac{3\delta}{4}<\epsilon_{1}+\epsilon_{2},\]
and \[
\left|\hat{x}-\bar{x}\right|\geq\left|\tilde{x}-x\right|-\left|x-\bar{x}\right|-\left|\hat{x}-\tilde{x}\right|\geq\epsilon_{1}-\frac{\delta}{2K}-\frac{\delta}{2}\geq\epsilon_{1}-\frac{3\delta}{4}>\epsilon_{1}-\epsilon_{2}.\]
Hence $\hat{x}\in\left(\mathbb{B}_{\epsilon_{1}+\epsilon_{2}}\left(\bar{x}\right)\backslash\mathbb{B}_{\epsilon_{1}-\epsilon_{2}}\left(\bar{x}\right)\right)\cap X$.
We now proceed to prove the inequality $\left|F\left(\tilde{x}\right)-F\left(\hat{x}\right)\right|<\kappa_{2}\left|\hat{x}-\tilde{x}\right|$
for the two cases.

Condition (i): Since $\left|\hat{x}-\tilde{x}\right|<\delta$, the
line segment $\left[\hat{x},\tilde{x}\right]$ lies in $\mathbb{B}_{\epsilon_{x}}\left(x\right)$
for some $x\in X$. Since the line segment $\left[\hat{x},\tilde{x}\right]$
is convex and $\lip\tilde{F}$ is bounded from above by $\kappa_{2}$
there, we have \begin{eqnarray*}
\left|F\left(\tilde{x}\right)-F\left(\hat{x}\right)\right| & = & \left|\tilde{F}\left(\tilde{x}\right)-\tilde{F}\left(\hat{x}\right)\right|\\
 & < & \kappa_{2}\left|\tilde{x}-\hat{x}\right|\end{eqnarray*}
 by \cite[Theorem 9.2]{RW98}. 

Condition (ii): The proof is similar, except that $\left[\hat{x},\tilde{x}\right]\subset X$,
and $\lipx F$ is bounded above by $\kappa_{2}$. 

On establishing $\left|F\left(\tilde{x}\right)-F\left(\hat{x}\right)\right|<\kappa_{2}\left|\hat{x}-\tilde{x}\right|$,
we note that \begin{eqnarray*}
F\left(\tilde{x}\right) & \in & F\left(\hat{x}\right)+\kappa_{2}\left|\hat{x}-\tilde{x}\right|\mathbb{B}\\
 & \subset & F_{\epsilon_{1}}\left(x^{\prime}\right)+\kappa_{2}\left|\hat{x}-\tilde{x}\right|\mathbb{B}\\
 & \subset & F_{\epsilon_{1}}\left(x^{\prime}\right)+K\kappa_{2}\left|x-x^{\prime}\right|\mathbb{B},\end{eqnarray*}
 and we are done. 
\end{proof}
 We are now ready to relate $\lipx\bar{f}_{\epsilon}\left(\bar{x}\right)$
to $\lipx f\left(\bar{x}\right)$. We remind the reader that in the
proof of Corollary \ref{cor:robust-regularization-lipschitz} below,
$f_{\epsilon}:X\rightrightarrows\mathbb{R}$ is a set-valued map as
introduced in Definition \ref{def:robust-regularization}, which is
similar to $\bar{f}_{\epsilon}$ but maps to intervals in $\mathbb{R}$.

\begin{cor}
\label{cor:robust-regularization-lipschitz}For $f:X\rightarrow\mathbb{R}$,
if either condition (i) or condition (ii) in Proposition \ref{pro:inf-lip-regularization}
for $F:X\rightarrow\mathbb{R}$ taken to be $f$ holds, then $\lipx\bar{f}_{\epsilon}\left(\bar{x}\right)<\infty$
for all small $\epsilon>0$. 
\end{cor}
\begin{proof}
By Proposition \ref{pro:inf-lip-regularization}, we have $\lip f_{\epsilon}\left(\bar{x}\right)<\infty$
with the given conditions. It remains to prove that $\lipx\bar{f}_{\epsilon}\left(\bar{x}\right)\leq\lip f_{\epsilon}\left(\bar{x}\right)$.
We can do this by proving that $\lipx\bar{S}\left(\bar{x}\right)\leq\lip S\left(\bar{x}\right)$,
where $S:X\rightrightarrows\mathbb{R}$ is a set-valued map, and $\bar{S}:X\rightarrow\mathbb{R}$
is defined by $\bar{S}\left(x\right)=\sup\left\{ y\mid y\in S\left(x\right)\right\} $.
Note that if $S=f_{\epsilon}$, then $\bar{S}=\overline{\left(f_{\epsilon}\right)}=\bar{f}_{\epsilon}$. 

For any $\kappa>\lip S\left(x\right)$, we have $\mathbf{d}\left(S\left(\tilde{x}\right),S\left(\hat{x}\right)\right)\leq\kappa\left|\tilde{x}-\hat{x}\right|$
for $\tilde{x},\hat{x}\in X$ close enough to $x$ by \cite[Definition 9.26]{RW98}.
The definition of the Pompeiu-Hausdorff distance tells us that $S\left(\tilde{x}\right)\subset S\left(\hat{x}\right)+\kappa\left|\tilde{x}-\hat{x}\right|$,
which implies $\bar{S}\left(\tilde{x}\right)\leq\bar{S}\left(\hat{x}\right)+\kappa\left|\tilde{x}-\hat{x}\right|$.
By reversing the roles of $\tilde{x}$ and $\hat{x}$, we obtain $\left|\bar{S}\left(\tilde{x}\right)-\bar{S}\left(\hat{x}\right)\right|\leq\kappa\left|\tilde{x}-\hat{x}\right|$.
So $\kappa>\lipx\bar{S}\left(x\right)$, and since $\kappa$ is arbitrary,
we have $\lipx\bar{S}\left(x\right)\leq\lip S\left(x\right)$ as needed.
\end{proof}

\section{\label{sec:Semialgebraic}Semi-algebraic robust regularization}

In this section, in Theorem \ref{thm:lip<infty}, we prove that if
$f:\mathbb{R}^{n}\rightarrow\mathbb{R}$ is continuous and semi-algebraic,
then at any given point, the robust regularization is locally Lipschitz
there for all sufficiently small $\epsilon>0$. This theorem is more
appealing than Corollary \ref{cor:robust-regularization-lipschitz}
because the required condition is weaker. The condition $\lip f\left(x\right)<\infty$
for all $x$ close to but not equal to $\bar{x}$ in Corollary \ref{cor:robust-regularization-lipschitz}
is a strong condition because if a function is not Lipschitz at a
point $\bar{x}$, it is likely that it is not Lipschitz at some points
close to $\bar{x}$ as well. For example in $f:\mathbb{R}^{2}\rightarrow\mathbb{R}$
defined by $f\left(x_{1},x_{2}\right)=\left|\sqrt{x_{1}}\right|$,
$f$ is not Lipschitz at all points where $x_{1}=0$.

We proceed to prove the main theorem of this section in the steps
outlined below.

\begin{prop}
\label{pro:G-semialgebraic}For $f:X\rightarrow\mathbb{R}$, where
$X\subset\mathbb{R}^{n}$ is convex, define $G:X\times\mathbb{R}_{+}\rightarrow\mathbb{R}_{+}\cup\left\{ \infty\right\} $
by \[
G\left(x,\epsilon\right):=\limsup_{\tilde{\epsilon}\rightarrow\epsilon}\lipx\bar{f}_{\tilde{\epsilon}}\left(x\right).\]
If $f$ is semi-algebraic, then the maps $\left(x,\epsilon\right)\mapsto\calmx\bar{f}_{\epsilon}\left(x\right)$,
$\left(x,\epsilon\right)\mapsto\lipx\bar{f}_{\epsilon}\left(x\right)$
and $G$ are semi-algebraic.
\end{prop}
\begin{proof}
The semi-algebraic nature is a consequence of the Tarski-Seidenberg
quantifier elimination.
\end{proof}
The semi-algebraicity of $\left(x,\epsilon\right)\mapsto\calm\bar{f}_{\epsilon}\left(x\right)$
gives us an indication of how the map $\epsilon\mapsto\calm\bar{f}_{\epsilon}\left(x\right)$
behaves asymptotically.

\begin{prop}
\label{pro:f-epsilon-calm}Suppose that $f:X\rightarrow\mathbb{R}$
is continuous and semi-algebraic, where $X\subset\mathbb{R}^{n}$.
Fix $x\in X$. Then $\calmx\bar{f}_{\epsilon}\left(x\right)=o\left(\frac{1}{\epsilon}\right)$
as $\epsilon\searrow0$. Hence $\bar{f}_{\epsilon}$ is calm at $x$
for all small $\epsilon>0$.
\end{prop}
\begin{proof}
The map $g_{x}$ is semi-algebraic because it can be written as a
composition of semi-algebraic maps $\epsilon\mapsto\left(x,\epsilon\right)\mapsto\bar{f}_{\epsilon}\left(x\right)$.
Thus $g_{x}$ is differentiable on some open interval of the form
$\left(0,\bar{\epsilon}\right)$ for $\bar{\epsilon}>0$. Recall that
$\calm g_{x}\left(\epsilon\right)=g_{x}^{\prime}\left(\epsilon\right)$
by Proposition \ref{pro:calm-eq-derivative}.

We show that for any $K>0$, we can reduce $\bar{\epsilon}$ if necessary
so that the map $\epsilon\mapsto\calm\bar{f}_{\epsilon}\left(x\right)$
is bounded from above by $\epsilon\mapsto\frac{K}{\epsilon}$ on $\epsilon\in\left[0,\bar{\epsilon}\right]$.
For any $K>0$, there exists an $\bar{\epsilon}>0$ such that either
$g_{x}^{\prime}\left(\epsilon\right)\leq\frac{K}{\epsilon}$ for all
$0<\epsilon<\bar{\epsilon}$, or $g_{x}^{\prime}\left(\epsilon\right)\geq\frac{K}{\epsilon}$
for all $0<\epsilon<\bar{\epsilon}$. The latter cannot happen, otherwise
for any $0<\epsilon<\bar{\epsilon}$,\begin{eqnarray*}
\bar{f}_{\epsilon}\left(x\right)-f\left(x\right) & = & \int_{0}^{\epsilon}g_{x}^{\prime}\left(s\right)ds\\
 & \geq & \int_{0}^{\epsilon}\frac{K}{s}ds=\infty.\end{eqnarray*}
This contradicts the continuity of $g_{x}$. If $\epsilon$ is small
enough, the derivatives of $g_{x}$ exist for all small $\epsilon>0$
and $g_{x}^{\prime}\left(\epsilon\right)=\calm\bar{f}_{\epsilon}\left(x\right)$
by Proposition \ref{pro:calm-eq-derivative}. This gives us the required
result.
\end{proof}
Consider $f:\left[0,1\right]\rightarrow\mathbb{R}$ defined by $f\left(x\right)=x^{1/k}$.
Then $g_{0}\left(\epsilon\right)=\epsilon^{1/k}$, so $\calm\bar{f}_{\epsilon}\left(0\right)=g_{0}^{\prime}\left(\epsilon\right)=\frac{1}{k}\epsilon^{\left(1/k\right)-1}$.
As $k\rightarrow\infty$, we see that the bound above is tight.

We are now ready to state the main theorem of this paper. In the particular
case of $X=\mathbb{R}^{n}$, we have the following theorem.

\begin{thm}
\label{thm:lip<infty}Consider any continuous semi-algebraic function
$f:\mathbb{R}^{n}\rightarrow\mathbb{R}$. At any fixed point $\bar{x}\in\mathbb{R}^{n}$,
the robust regularization $\bar{f}_{\epsilon}$ is Lipschitz at $\bar{x}$,
and its calmness and Lipschitz moduli, $\calm\bar{f}_{\epsilon}\left(\bar{x}\right)$
and $\lip\bar{f}_{\epsilon}\left(\bar{x}\right)$, agree for sufficiently
small $\epsilon$ and behave like $o\left(\frac{1}{\epsilon}\right)$
as $\epsilon\downarrow0$. 
\end{thm}
\begin{proof}
In view of Proposition \ref{pro:f-epsilon-calm}, we only need to
prove the there is some $\bar{\epsilon}>0$ such that $\lip\bar{f}_{\epsilon}\left(\bar{x}\right)=\calm\bar{f}_{\epsilon}\left(\bar{x}\right)$
for all $\epsilon\in(0,\bar{\epsilon}]$. We can assume that $g_{\bar{x}}$
is twice continuously differentiable in $(0,\bar{\epsilon}]$. The
graph of $G:\mathbb{R}^{n}\times\mathbb{R}_{+}\rightarrow\mathbb{R}_{+}$
as defined in Proposition \ref{pro:G-semialgebraic} is semi-algebraic,
so by the decomposition theorem \cite[Theorem 6.7]{C99}, there is
a finite partition of definable $\mathcal{C}^{2}$ manifolds $C_{1},\dots,C_{l}$
such that $G\mid_{C_{i}}$ is $\mathcal{C}^{2}$.

If the segment $\left\{ \bar{x}\right\} \times(0,\bar{\epsilon}]$
lies in the (relative) interior of one definable manifold, then \begin{eqnarray*}
\lip\bar{f}_{\epsilon}\left(\bar{x}\right) & = & \limsup_{\tilde{x}\rightarrow\bar{x}}\calm\bar{f}_{\epsilon}\left(\tilde{x}\right)\mbox{ (by Proposition }\ref{pro:lip-eq-cl-calm}\mbox{)}\\
 & = & \limsup_{\tilde{x}\rightarrow\bar{x}}g_{\tilde{x}}^{\prime}\left(\epsilon\right)\mbox{ (by Proposition }\ref{pro:calm-eq-derivative}\mbox{)}\\
 & = & g_{\bar{x}}^{\prime}\left(\epsilon\right)\\
 & = & \calm\bar{f}_{\epsilon}\left(\bar{x}\right),\end{eqnarray*}
 and we have nothing to do. Therefore, assume that the segment is
on the boundary of two or more of the $C_{i}$.

Since $G$ is semi-algebraic, the map $\epsilon\mapsto\limsup_{\alpha\rightarrow\epsilon}\lip\bar{f}_{\alpha}\left(\bar{x}\right)$
is semi-algebraic, so we can reduce $\bar{\epsilon}>0$ as necessary
such that either

(1) $\limsup_{\alpha\rightarrow\epsilon}\lip\bar{f}_{\alpha}\left(\bar{x}\right)<\calm\bar{f}_{\epsilon}\left(\bar{x}\right)$
for all $\epsilon\in(0,\bar{\epsilon}]$, or

(2) $\limsup_{\alpha\rightarrow\epsilon}\lip\bar{f}_{\alpha}\left(\bar{x}\right)=\calm\bar{f}_{\epsilon}\left(\bar{x}\right)$
for all $\epsilon\in(0,\bar{\epsilon}]$, or

(3) $\limsup_{\alpha\rightarrow\epsilon}\lip\bar{f}_{\alpha}\left(\bar{x}\right)>\calm\bar{f}_{\epsilon}\left(\bar{x}\right)$
for all $\epsilon\in(0,\bar{\epsilon}]$.

Case (1) cannot hold because $\lip\bar{f}_{\epsilon}\left(\bar{x}\right)\geq\calm\bar{f}_{\epsilon}\left(\bar{x}\right)$.
Case (2) is what we seek to prove, so we proceed to show that case
(3) cannot happen by contradiction.

We can choose $\tilde{\epsilon},M_{1},M_{2}>0$ such that $0<\tilde{\epsilon}<\bar{\epsilon}$
and \[
\calm\bar{f}_{\epsilon}\left(\bar{x}\right)<M_{2}<M_{1}<\limsup_{\alpha\rightarrow\epsilon}\lip\bar{f}_{\alpha}\left(\bar{x}\right)\mbox{ for all }\epsilon\in\left[\tilde{\epsilon},\bar{\epsilon}\right].\]
We state and prove a lemma important to the rest of the proof before
continuing.
\begin{lem}
\label{lem:5-conditions}There exists an interval $\left(\epsilon_{1},\epsilon_{2}\right)$
contained in $(\tilde{\epsilon},\bar{\epsilon}]$ and a manifold $T_{1}\subset\mathbb{R}^{n}\times\mathbb{R}_{+}$
such that

(1) $\left\{ \bar{x}\right\} \times\left(\epsilon_{1},\epsilon_{2}\right)\subset\cl\left(T_{1}\right)$.

(2) $T_{1}$ is an open $\mathcal{C}^{2}$ manifold. 

(3) $H:\mathbb{R}^{n}\times\mathbb{R}_{+}\rightarrow\mathbb{R}$,
defined by $H\left(x,\epsilon\right)=\bar{f}_{\epsilon}\left(x\right)$,
is $\mathcal{C}^{2}$ in $T_{1}$.

(4) For all $\left(x,\epsilon\right)\in T_{1}$, we have $M_{1}\leq g_{x}^{\prime}\left(\epsilon\right)<\infty$.

(5) $\left(x,\epsilon\right)\mapsto g_{x}^{\prime}\left(\epsilon\right)$
is continuous in $T_{1}$.
\end{lem}
\begin{proof}
Consider the set \[
T:=\left\{ \left(x,\epsilon\right)\mid M_{1}\leq g_{x}^{\prime}\left(\epsilon\right)<\infty\right\} .\]
First, we prove that $\left\{ \bar{x}\right\} \times\left[\tilde{\epsilon},\bar{\epsilon}\right]\subset\mbox{cl }T$.
It suffices to show that for all $\epsilon\in(\tilde{\epsilon},\bar{\epsilon}]$,
$\left(\bar{x},\epsilon\right)\in\mbox{cl }T$. This can in turn be
proven by showing that for all $\delta>0$, we can find $x^{\prime},\epsilon^{\prime}$
such that $\left|\bar{x}-x^{\prime}\right|<\delta$, $\left|\epsilon-\epsilon^{\prime}\right|<\delta$
such that $\left(x^{\prime},\epsilon^{\prime}\right)\in T$, or equivalently,
$M_{1}\leq g_{x^{\prime}}^{\prime}\left(\epsilon^{\prime}\right)<\infty$.

Since $\limsup_{\alpha\rightarrow\epsilon}\lip\bar{f}_{\alpha}\left(\bar{x}\right)>M_{1}$,
there is some $\epsilon^{\circ}$ such that $\left|\epsilon^{\circ}-\epsilon\right|<\frac{\delta}{2}$
and $\lip\bar{f}_{\epsilon^{\circ}}\left(\bar{x}\right)>M_{1}$.

Next, since\[
\limsup_{x\rightarrow\bar{x}}\left|\partial g_{x}\left(\epsilon^{\circ}\right)\right|\geq\limsup_{x\rightarrow\bar{x}}\calm\bar{f}_{\epsilon^{\circ}}\left(x\right)=\lip\bar{f}_{\epsilon^{\circ}}\left(\bar{x}\right),\]
there is some $x^{\prime}$ such that $\left|\bar{x}-x^{\prime}\right|<\delta$
and $\left|\partial g_{x^{\prime}}\left(\epsilon^{\circ}\right)\right|>\frac{1}{2}\lip\bar{f}_{\epsilon^{\circ}}\left(\bar{x}\right)+\frac{1}{2}M_{1}$. 

Finally, since $g_{x^{\prime}}\left(\cdot\right)$ is semi-algebraic,
we can find some $\epsilon^{\prime}$ such that $\left|\epsilon^{\prime}-\epsilon^{\circ}\right|<\frac{\delta}{2}$,
$g_{x^{\prime}}^{\prime}\left(\epsilon^{\prime}\right)$ is well defined
and finite, and \[
g_{x^{\prime}}^{\prime}\left(\epsilon^{\prime}\right)>\left|\partial g_{x^{\prime}}\left(\epsilon^{\circ}\right)\right|-\frac{1}{2}\left(\lip\bar{f}_{\epsilon^{\circ}}\left(\bar{x}\right)-M_{1}\right)>M_{1}.\]
This choice of $x^{\prime}$ and $\epsilon^{\prime}$ are easily verified
to satisfy the requirements stated.

By the decomposition theorem \cite[Theorem 6.7]{C99}, $T$ can be
decomposed into a finite disjoint union of $\mathcal{C}^{2}$ smooth
manifolds $T_{1},T_{2},\dots,T_{p}$ on which $H$ is $\mathcal{C}^{2}$.
Since $\left\{ \bar{x}\right\} \times\left[\tilde{\epsilon},\bar{\epsilon}\right]\subset\mbox{cl }T$,
there must be some $T_{i}$ and $\left(\epsilon_{1},\epsilon_{2}\right)$
such that $\left\{ \bar{x}\right\} \times\left(\epsilon_{1},\epsilon_{2}\right)\subset\mbox{cl }T_{i}$.
Without loss of generality, let one such $T_{i}$ be $T_{1}$.

Conditions (1), (2), (3) and (4) are automatically satisfied. Note
that $g_{x}^{\prime}\left(\epsilon\right)$ is exactly the derivative
of $H\left(\cdot,\cdot\right)$ with respect to the second coordinate,
and so Property (5) is satisfied. This concludes the proof of the
lemma.
\end{proof}
We now continue with the rest of the proof of the theorem. If $T_{1}$
is of dimension one, then we have $T_{1}\supset\left\{ \bar{x}\right\} \times\left(\epsilon_{1},\epsilon_{2}\right)$.
Recall that if the derivative $g_{\bar{x}}^{\prime}\left(\epsilon\right)$
exists, then $g_{\bar{x}}^{\prime}\left(\epsilon\right)=\calm\bar{f}_{\epsilon}\left(\bar{x}\right)$
by Proposition \ref{pro:calm-eq-derivative}. This would mean that
$\calm\bar{f}_{\epsilon}\left(\bar{x}\right)\geq M_{2}$, which contradicts
our earlier assumption of $\calm\bar{f}_{\epsilon}\left(\bar{x}\right)<M_{2}$.
Therefore, the manifold $T_{1}$ is of dimension at least two. 

Using Lemma \ref{lem:continuous-curve-selection} which we will prove
later, we can construct the map $\varphi:[0,1)\times\left(\hat{\epsilon}_{1},\hat{\epsilon}_{2}\right)\rightarrow\mbox{cl}T_{1}$,
such that its derivative with respect to the second variable exists
and is continuous, and $\varphi\left(0,\epsilon\right)=\left(\bar{x},\epsilon\right)$
for all $\epsilon\in\left(\hat{\epsilon}_{1},\hat{\epsilon}_{2}\right)$.

For each $0<\delta<1$, consider the path $\tilde{x}_{\delta}:\left[\hat{\epsilon}_{1},\hat{\epsilon}_{2}\right]\rightarrow\mathbb{R}^{n}$
defined by $\tilde{x}_{\delta}\left(\epsilon\right):=\varphi\left(\delta,\epsilon\right)$.
We have\begin{eqnarray*}
 &  & \bar{f}_{\hat{\epsilon}_{2}}\left(\tilde{x}_{\delta}\left(\hat{\epsilon}_{2}\right)\right)-\bar{f}_{\hat{\epsilon}_{1}}\left(\tilde{x}_{\delta}\left(\hat{\epsilon}_{1}\right)\right)\\
 & = & \int_{\hat{\epsilon}_{1}}^{\hat{\epsilon}_{2}}\nabla H\left(\tilde{x}_{\delta}\left(s\right),s\right)\cdot\left(\tilde{x}_{\delta}^{\prime}\left(s\right),1\right)ds\\
 & = & \int_{\hat{\epsilon}_{1}}^{\hat{\epsilon}_{2}}\nabla_{x}H\left(\tilde{x}_{\delta}\left(s\right),s\right)\cdot\tilde{x}_{\delta}^{\prime}\left(s\right)ds+\int_{\hat{\epsilon}_{1}}^{\hat{\epsilon}_{2}}\nabla_{s}H\left(\tilde{x}_{\delta}\left(s\right),s\right)ds,\end{eqnarray*}
where $H\left(x,\epsilon\right)=\bar{f}_{\epsilon}\left(x\right)$.
The second component of $\nabla H\left(\tilde{x}_{\delta}\left(s\right),s\right)$
is simply $g_{\tilde{x}_{\delta}\left(s\right)}^{\prime}\left(s\right)$.
The first component can be analyzed as follows: \begin{eqnarray*}
 &  & \nabla_{x}H\left(\tilde{x}_{\delta}\left(s\right),s\right)\cdot\tilde{x}_{\delta}^{\prime}\left(s\right)\\
 & = & \lim_{t\rightarrow0}\frac{1}{t}\left(H\left(\tilde{x}_{\delta}\left(s\right)+t\tilde{x}_{\delta}^{\prime}\left(s\right),s\right)-H\left(\tilde{x}_{\delta}\left(s\right),s\right)\right)\\
 & = & \lim_{t\rightarrow0}\frac{1}{t}\left(\bar{f}_{s}\left(\tilde{x}_{\delta}\left(s\right)+t\tilde{x}_{\delta}^{\prime}\left(s\right)\right)-\bar{f}_{s}\left(\tilde{x}_{\delta}\left(s\right)\right)\right).\end{eqnarray*}
Provided that $t\left|\tilde{x}_{\delta}^{\prime}\left(s\right)\right|<s$,
$\mathbb{B}_{s-t\left|\tilde{x}_{\delta}^{\prime}\left(s\right)\right|}\left(\tilde{x}_{\delta}\left(s\right)\right)\subset\mathbb{B}_{s}\left(\tilde{x}_{\delta}\left(s\right)+t\tilde{x}_{\delta}^{\prime}\left(s\right)\right)$,
and so\begin{eqnarray*}
 &  & \nabla_{x}H\left(\tilde{x}_{\delta}\left(s\right),s\right)\cdot\tilde{x}_{\delta}^{\prime}\left(s\right)\\
 & \geq & \lim_{t\rightarrow0}\frac{1}{t}\left(\bar{f}_{s-t\left|\tilde{x}_{\delta}^{\prime}\left(s\right)\right|}\left(\tilde{x}_{\delta}\left(s\right)\right)-\bar{f}_{s}\left(\tilde{x}_{\delta}\left(s\right)\right)\right)\\
 & = & \left|\tilde{x}_{\delta}^{\prime}\left(s\right)\right|\lim_{t\rightarrow0}\frac{1}{t\left|\tilde{x}_{\delta}^{\prime}\left(s\right)\right|}\left(\bar{f}_{s-t\left|\tilde{x}_{\delta}^{\prime}\left(s\right)\right|}\left(\tilde{x}_{\delta}\left(s\right)\right)-\bar{f}_{s}\left(\tilde{x}_{\delta}\left(s\right)\right)\right)\\
 & = & -\left|\tilde{x}_{\delta}^{\prime}\left(s\right)\right|g_{\tilde{x}_{\delta}\left(s\right)}^{\prime}\left(s\right).\end{eqnarray*}
Hence, \begin{eqnarray*}
 &  & \bar{f}_{\hat{\epsilon}_{2}}\left(\tilde{x}_{\delta}\left(\hat{\epsilon}_{2}\right)\right)-\bar{f}_{\hat{\epsilon}_{1}}\left(\tilde{x}_{\delta}\left(\hat{\epsilon}_{1}\right)\right)\\
 & = & \int_{\hat{\epsilon}_{1}}^{\hat{\epsilon}_{2}}\nabla_{x}H\left(\tilde{x}_{\delta}\left(s\right),s\right)\cdot\tilde{x}_{\delta}^{\prime}\left(s\right)ds+\int_{\hat{\epsilon}_{1}}^{\hat{\epsilon}_{2}}\nabla_{s}H\left(\tilde{x}_{\delta}\left(s\right),s\right)ds\\
 & \geq & \int_{\hat{\epsilon}_{1}}^{\hat{\epsilon}_{2}}\left(1-\left|\tilde{x}_{\delta}^{\prime}\left(s\right)\right|\right)g_{\tilde{x}_{\delta}\left(s\right)}^{\prime}\left(s\right)ds.\end{eqnarray*}
Since the derivatives of $\varphi$ are continuous, $\tilde{x}_{\delta}^{\prime}\left(s\right)\rightarrow\tilde{x}_{0}^{\prime}\left(s\right)=0$
as $\delta\rightarrow0$ for $\hat{\epsilon}_{1}<s<\hat{\epsilon}_{2}$.
In fact, the term $\left|\tilde{x}_{\delta}^{\prime}\left(s\right)\right|$
converges to zero uniformly in $\left[\hat{\epsilon}_{1},\hat{\epsilon}_{2}\right]$.
To see this, recall that $\tilde{x}_{\delta}^{\prime}\left(s\right)$
is a partial derivative of $\varphi$. Since $\varphi$ is $\mathcal{C}^{1}$,
$\tilde{x}_{\delta}^{\prime}\left(s\right)$ is continuous with respect
to $s$ and $\delta$. For any $\beta>0$ and $s\in\left[\hat{\epsilon}_{1},\hat{\epsilon}_{2}\right]$,
there exists $\gamma_{s}$ such that \[
\left|\tilde{x}_{\delta}^{\prime}\left(\tilde{s}\right)\right|<\beta\mbox{ if }\delta<\gamma_{s}\mbox{ and }\left|\tilde{s}-s\right|<\gamma_{s}.\]
The existence of $\gamma$ such that \[
\left|\tilde{x}_{\delta}^{\prime}\left(s\right)\right|<\beta\mbox{ if }\delta<\gamma\mbox{ and }s\in\left[\hat{\epsilon}_{1},\hat{\epsilon}_{2}\right]\]
follows by the compactness of $\left[\hat{\epsilon}_{1},\hat{\epsilon}_{2}\right]$.
So we may choose $\delta$ small enough so that \[
\left(1-\left|\tilde{x}_{\delta}^{\prime}\left(s\right)\right|\right)>\frac{M_{1}+M_{2}}{2M_{1}}\mbox{ for all }s\in\left[\hat{\epsilon}_{1},\hat{\epsilon}_{2}\right].\]

Now, for $\delta$ small enough and $i=1,2$, we have $g_{\bar{x}}^{\prime}\left(\hat{\epsilon}_{i}\right)<M_{2}$,
so this gives us $\calm\bar{f}_{\hat{\epsilon}_{i}}\left(\bar{x}\right)=g_{\bar{x}}^{\prime}\left(\hat{\epsilon}_{i}\right)<M_{2}$
by Proposition \ref{pro:calm-eq-derivative}. Therefore, if $\delta$
is small enough, \[
\left|\bar{f}_{\hat{\epsilon}_{i}}\left(\tilde{x}_{\delta}\left(\hat{\epsilon}_{i}\right)\right)-\bar{f}_{\hat{\epsilon}_{i}}\left(\bar{x}\right)\right|\leq M_{2}\left|\tilde{x}_{\delta}\left(\hat{\epsilon}_{i}\right)-\bar{x}\right|.\]

Recall that if the derivative $g_{\bar{x}}^{\prime}\left(\epsilon\right)$
exists, then $g_{\bar{x}}^{\prime}\left(\epsilon\right)=\calm\bar{f}_{\epsilon}\left(\bar{x}\right)$
by Proposition \ref{pro:calm-eq-derivative}. On the one hand,  we
have \[
\bar{f}_{\hat{\epsilon}_{2}}\left(\bar{x}\right)-\bar{f}_{\hat{\epsilon}_{1}}\left(\bar{x}\right)=\int_{\hat{\epsilon}_{1}}^{\hat{\epsilon}_{2}}g_{\bar{x}}^{\prime}\left(s\right)ds\leq\int_{\hat{\epsilon}_{1}}^{\hat{\epsilon}_{2}}M_{2}ds=M_{2}\left(\hat{\epsilon}_{2}-\hat{\epsilon}_{1}\right).\]
But on the other hand, $\tilde{x}_{\delta}\left(s\right)\in T_{1}$
for $0<\delta<1$, and so $g_{\tilde{x}_{\delta}\left(s\right)}^{\prime}\left(s\right)\geq M_{1}$
by Lemma \ref{lem:5-conditions}. If $\delta$ is small enough, we
have \begin{eqnarray*}
 &  & \left|\bar{f}_{\hat{\epsilon}_{2}}\left(\bar{x}\right)-\bar{f}_{\hat{\epsilon}_{1}}\left(\bar{x}\right)\right|\\
 & \geq & \left|\bar{f}_{\hat{\epsilon}_{2}}\left(\tilde{x}_{\delta}\left(\hat{\epsilon}_{2}\right)\right)-\bar{f}_{\hat{\epsilon}_{1}}\left(\tilde{x}_{\delta}\left(\hat{\epsilon}_{1}\right)\right)\right|\\
 &  & \,\,\,\,\,\,-\left(\left|\bar{f}_{\hat{\epsilon}_{2}}\left(\tilde{x}_{\delta}\left(\hat{\epsilon}_{2}\right)\right)-\bar{f}_{\hat{\epsilon}_{2}}\left(\bar{x}\right)\right|+\left|\bar{f}_{\hat{\epsilon}_{1}}\left(\tilde{x}_{\delta}\left(\hat{\epsilon}_{1}\right)\right)-\bar{f}_{\hat{\epsilon}_{1}}\left(\bar{x}\right)\right|\right)\\
 & \geq & \int_{\hat{\epsilon}_{1}}^{\hat{\epsilon}_{2}}\left(1-\left|\tilde{x}_{\delta}^{\prime}\left(s\right)\right|\right)g_{\tilde{x}_{\delta}\left(s\right)}^{\prime}\left(s\right)ds\\
 &  & \,\,\,\,\,\,-M_{2}\left(\left|\tilde{x}_{\delta}\left(\hat{\epsilon}_{2}\right)-\bar{x}\right|+\left|\tilde{x}_{\delta}\left(\hat{\epsilon}_{1}\right)-\bar{x}\right|\right)\\
 & \geq & \int_{\hat{\epsilon}_{1}}^{\hat{\epsilon}_{2}}\left(1-\left|\tilde{x}_{\delta}^{\prime}\left(s\right)\right|\right)M_{1}ds-M_{2}\left(\left|\tilde{x}_{\delta}\left(\hat{\epsilon}_{2}\right)-\bar{x}\right|+\left|\tilde{x}_{\delta}\left(\hat{\epsilon}_{1}\right)-\bar{x}\right|\right)\\
 & \geq & \int_{\hat{\epsilon}_{1}}^{\hat{\epsilon}_{2}}\left(\frac{M_{1}+M_{2}}{2}\right)ds-M_{2}\left(\left|\tilde{x}_{\delta}\left(\hat{\epsilon}_{2}\right)-\bar{x}\right|+\left|\tilde{x}_{\delta}\left(\hat{\epsilon}_{1}\right)-\bar{x}\right|\right)\\
 & = & \left(\frac{M_{1}+M_{2}}{2}\right)\left(\hat{\epsilon}_{2}-\hat{\epsilon}_{1}\right)-M_{2}\left(\left|\tilde{x}_{\delta}\left(\hat{\epsilon}_{2}\right)-\bar{x}\right|+\left|\tilde{x}_{\delta}\left(\hat{\epsilon}_{1}\right)-\bar{x}\right|\right).\end{eqnarray*}
As $\delta$ is arbitrarily small and the terms $\left|\tilde{x}_{\delta}\left(\hat{\epsilon}_{i}\right)-\bar{x}\right|\rightarrow0$
as $\delta\rightarrow0$ for $i=1,2$, we have $\left|\bar{f}_{\hat{\epsilon}_{2}}\left(\bar{x}\right)-\bar{f}_{\hat{\epsilon}_{1}}\left(\bar{x}\right)\right|\geq\left(\frac{M_{1}+M_{2}}{2}\right)\left(\hat{\epsilon}_{2}-\hat{\epsilon}_{1}\right)$.
This is a contradiction, and thus we are done.
\end{proof}
Before we prove Lemma \ref{lem:continuous-curve-selection} below,
we need to recall the definition of simplicial complexes from \cite[Section 3.2.1]{C02}.
A \emph{simplex} with vertices $a_{0},\dots,a_{d}$ is \begin{eqnarray*}
\left[a_{0},\dots,a_{d}\right] & = & \{ x\in\mathbb{R}^{n}\mid\exists\lambda_{0},\dots,\lambda_{d}\in\left[0,1\right],\\
 &  & \qquad\qquad\qquad\sum_{i=0}^{d}\lambda_{i}=1\mbox{ and }x=\sum_{i=0}^{d}\lambda_{i}a_{i}.\}\end{eqnarray*}
The corresponding \emph{open simplex} is\begin{eqnarray*}
\left(a_{0},\dots,a_{d}\right) & = & \{ x\in\mathbb{R}^{n}\mid\exists\lambda_{0},\dots,\lambda_{d}\in\left(0,1\right),\\
 &  & \qquad\qquad\qquad\sum_{i=0}^{d}\lambda_{i}=1\mbox{ and }x=\sum_{i=0}^{d}\lambda_{i}a_{i}.\}\end{eqnarray*}
We shall denote by $\mbox{int}\left(\sigma\right)$ the open simplex
corresponding to the simplex $\sigma$. A face of the simplex $\sigma=\left[a_{0},\dots,a_{d}\right]$
is a simplex $\tau=\left[b_{0},\dots,b_{e}\right]$ such that \[
\left\{ b_{0},\dots,b_{e}\right\} \subset\left\{ a_{0},\dots,a_{d}\right\} .\]

A \emph{finite simplicial complex} in $\mathbb{R}^{n}$ is a finite
collection $K=\left\{ \sigma_{1},\dots,\sigma_{p}\right\} $ of simplices
$\sigma_{i}\subset\mathbb{R}^{n}$ such that, for every $\sigma_{i},\sigma_{j}\in K$,
the intersection $\sigma_{i}\cap\sigma_{j}$ is either empty or is
a common face of $\sigma_{i}$ and $\sigma_{j}$. We set $\left|K\right|=\cup_{\sigma_{i}\in K}\sigma_{i}$;
this is a semi-algebraic subset of $\mathbb{R}^{n}$. We recall a
result on relating semi-algebraic sets to simplicial complexes.

\begin{thm}
\label{thm:[C02,3.12]}\cite[Theorem 3.12]{C02} Let $S\subset\mathbb{R}^{n}$
be a compact semi-algebraic set, and $S_{1},\dots,S_{p}$, semi-algebraic
subsets of $S$. Then there exists a finite simplicial complex $K$
in $\mathbb{R}^{n}$ and a semi-algebraic homeomorphism $h:\left|K\right|\rightarrow S$,
such that each $S_{k}$ is the image by $h$ of a union of open simplices
of $K$. 
\end{thm}
We need yet another result for the proof of Lemma \ref{lem:continuous-curve-selection}.

\begin{prop}
\label{pro:continuity-proposition}Suppose that $\phi:\left(0,1\right)^{2}\rightarrow\mathbb{R}$,
not necessarily semi-algebraic, is continuous in $\left(0,1\right)^{2}$.
Let $\gph\phi\subset\left(0,1\right)^{2}\times\mathbb{R}$ be the
graph of $\phi$. Then for any $t\in\left(0,1\right)$, $\cl\left(\gph\phi\right)\cap\left(0,t\right)\times\mathbb{R}$
is either a single point or a connected line segment.
\end{prop}
\begin{proof}
Suppose that $\left(\left(0,t\right),a_{1}\right)$ and $\left(\left(0,t\right),a_{2}\right)$
lie in $\mbox{cl}\left(\gph\phi\right)$. We need to show that for
any $\alpha\in\left(a_{1},a_{2}\right)$, $\left(\left(0,t\right),\alpha\right)$
lies in $\mbox{cl}\left(\gph\phi\right)$.

For any $\epsilon>0$, we can find points $p_{1},p_{2}\in\left(0,1\right)^{2}$
such that the points $\left(p_{1},\tilde{a}_{1}\right),\left(p_{2},\tilde{a}_{2}\right)\in\gph\phi$
are such that $\left|\tilde{a}_{i}-a_{i}\right|<\epsilon$ and $\left|p_{i}-\left(0,t\right)\right|<\epsilon$
for $i=1,2$. Recall that by definition $\tilde{a}_{i}=\phi\left(p_{i}\right)$
for $i=1,2$. Choose $\epsilon$ such that $\tilde{a}_{1}+\epsilon<\tilde{a}_{2}-\epsilon$.
By the intermediate value theorem, for any $\alpha\in\left(\tilde{a}_{1}+\epsilon,\tilde{a}_{2}-\epsilon\right)$,
there exists a point $p$ in the line segment $\left[p_{1},p_{2}\right]$
such that $\phi\left(p\right)=\alpha$. Moreover, $\left|p-\left(0,t\right)\right|<\max_{i=1,2}\left|p_{i}-\left(0,t\right)\right|$.
Letting $\epsilon\rightarrow0$, we see that $\left(\left(0,t\right),\alpha\right)\in\cl\left(\gph\phi\right)$
as needed.
\end{proof}
We now prove our last result important for the proof of Theorem \ref{thm:lip<infty}.
The proof of the lemma below is similar to the proof of the Curve
Selection Lemma in \cite[Theorem 3.13]{C02}.

\begin{lem}
\label{lem:continuous-curve-selection} Let $S\subset\mathbb{R}^{n}$
be a semi-algebraic set, and $\tau:\left[\epsilon_{1},\epsilon_{2}\right]\rightarrow\mathbb{R}^{n}$
be a semi-algebraic curve such that $\tau\left(\left[\epsilon_{1},\epsilon_{2}\right]\right)\cap S=\emptyset$
and $\tau\left(\left[\epsilon_{1},\epsilon_{2}\right]\right)\subset\cl\left(S\right)$.
Then there exists a function $\varphi:\left[0,1\right]\times\left[\hat{\epsilon}_{1},\hat{\epsilon}_{2}\right]\rightarrow\mathbb{R}^{n}$,
with $\left[\hat{\epsilon}_{1},\hat{\epsilon}_{2}\right]\neq\emptyset$
and $\left[\hat{\epsilon}_{1},\hat{\epsilon}_{2}\right]\subset\left[\epsilon_{1},\epsilon_{2}\right]$,
such that 

(1) $\varphi\left(0,\epsilon\right)=\tau\left(\epsilon\right)$ for
$\epsilon\in\left[\hat{\epsilon}_{1},\hat{\epsilon}_{2}\right]$ and
$\varphi\left((0,1]\times\left[\hat{\epsilon}_{1},\hat{\epsilon}_{2}\right]\right)\subset S$. 

(2) The partial derivative of $\varphi$ with respect to the second
variable, which we denote by $\frac{\partial}{\partial\epsilon}\varphi$,
exists and is continuous in $\left[0,1\right]\times\left[\hat{\epsilon}_{1},\hat{\epsilon}_{2}\right]$.
\end{lem}
\begin{proof}
Replacing $S$ with its intersection with a closed bounded set containing
$\tau\left(\left[\epsilon_{1},\epsilon_{2}\right]\right)$, we can
assume $S$ is bounded. Then $\mbox{cl}\left(S\right)$ is a compact
semi-algebraic set. By Theorem \ref{thm:[C02,3.12]}, there is a finite
simplicial complex $K$ and a semi-algebraic homeomorphism $h:\left|K\right|\rightarrow\mbox{cl}\left(S\right)$,
such that $S$ and $\tau\left(\left[\epsilon_{1},\epsilon_{2}\right]\right)$
are images by $h$ of a union of open simplices in $K$ . In particular,
this means that there is an open interval $\left(\hat{\epsilon}_{1},\hat{\epsilon}_{2}\right)\subset\left[\epsilon_{1},\epsilon_{2}\right]$
such that $\tau\left(\left(\hat{\epsilon}_{1},\hat{\epsilon}_{2}\right)\right)$
is an image by $h$ of a 1-dimensional open simplex in $K$. Since
$h^{-1}\circ\tau\left(\left(\hat{\epsilon}_{1},\hat{\epsilon}_{2}\right)\right)$
is in $\mbox{cl}\left(S\right)$ but not in $S$, there is a simplex
$\sigma$ of $K$ which has $h^{-1}\circ\tau\left(\left[\hat{\epsilon}_{1},\hat{\epsilon}_{2}\right]\right)$
lying in the boundary of $\sigma$, and $h\left(\mbox{int}\left(\sigma\right)\right)\subset S$. 

Let $\hat{\sigma}$ be the barycenter of $\sigma$. Define the map
$\delta:\left[0,1\right]\times\left[\hat{\epsilon}_{1},\hat{\epsilon}_{2}\right]\rightarrow\mathbb{R}^{n}$
by \[
\delta\left(t,\epsilon\right)=\left(1-t\right)h^{-1}\circ\tau\left(\epsilon\right)+t\hat{\sigma}.\]
The map above satisfies $\delta\left((0,1]\times\left(\hat{\epsilon}_{1},\hat{\epsilon}_{2}\right)\right)\subset\mbox{int}\left(\sigma\right)$.
By contracting the interval $\left[\hat{\epsilon}_{1},\hat{\epsilon}_{2}\right]$
slightly, $\varphi=h\circ\delta$ satisfies property (1). 

By contracting the interval $\left[\hat{\epsilon}_{1},\hat{\epsilon}_{2}\right]$
if necessary and applying the decomposition theorem \cite[Theorem 6.7]{C99},
we can assume that $\varphi$ is $\mathcal{C}^{1}$ in the set $(0,\bar{t}]\times\left[\hat{\epsilon}_{1},\hat{\epsilon}_{2}\right]$
for some $\bar{t}\in\left(0,1\right)$.

Since $\tau$ is semi-algebraic, we contract the interval $\left[\hat{\epsilon}_{1},\hat{\epsilon}_{2}\right]$
again if necessary so that $\tau$ is $\mathcal{C}^{1}$ there. Therefore,
$\frac{\partial}{\partial\epsilon}\varphi$ exists in $\left[0,\bar{t}\right]\times\left[\hat{\epsilon}_{1},\hat{\epsilon}_{2}\right]$.
It remains to show that $\frac{\partial}{\partial\epsilon}\varphi$
is continuous in $\left[0,\bar{t}\right]\times\left[\hat{\epsilon}_{1},\hat{\epsilon}_{2}\right]$.
We do this by showing that $\frac{\partial}{\partial\epsilon}\varphi_{i}:\left[0,\bar{t}\right]\times\left[\hat{\epsilon}_{1},\hat{\epsilon}_{2}\right]\rightarrow\mathbb{R}$,
the $i$th component of the derivative with respect to the second
variable, is continuous for each $i$.

Since $\frac{\partial}{\partial\epsilon}\varphi_{i}$ is continuous
in $(0,\bar{t}]\times\left[\hat{\epsilon}_{1},\hat{\epsilon}_{2}\right]$,
it remains to show that it is continuous at every point in $\left\{ 0\right\} \times\left[\hat{\epsilon}_{1},\hat{\epsilon}_{2}\right]$.
The graph of $\frac{\partial}{\partial\epsilon}\varphi_{i}$ corresponding
to the domain $(0,\bar{t}]\times\left[\hat{\epsilon}_{1},\hat{\epsilon}_{2}\right]$,
which we denote by $\gph\left(\frac{\partial}{\partial\epsilon}\varphi_{i}\right)$,
is a subset of $(0,\bar{t}]\times\left[\hat{\epsilon}_{1},\hat{\epsilon}_{2}\right]\times\mathbb{R}$.
We show that $\left(\left(0,\epsilon\right),\frac{\partial}{\partial\epsilon}\varphi_{i}\left(0,\epsilon\right)\right)\in\cl\left(\gph\left(\frac{\partial}{\partial\epsilon}\varphi_{i}\right)\right)$.
For small $t_{1},t_{2}>0$, consider $\varphi_{i}\left(t_{1},\epsilon-t_{2}\right)$
and $\varphi_{i}\left(t_{1},\epsilon+t_{2}\right)$. By the intermediate
value theorem, there is some $\tilde{\epsilon}\in\left(\epsilon-t_{2},\epsilon+t_{2}\right)$
such that \[
\frac{\partial}{\partial\epsilon}\varphi_{i}\left(t_{1},\tilde{\epsilon}\right)=\frac{1}{2t_{2}}\left(\varphi_{i}\left(t_{1},\epsilon+t_{2}\right)-\varphi_{i}\left(t_{1},\epsilon-t_{2}\right)\right).\]
If $t_{2}$ were chosen such that \[
\left|\frac{1}{2t_{2}}\left(\varphi_{i}\left(0,\epsilon+t_{2}\right)-\varphi_{i}\left(0,\epsilon-t_{2}\right)\right)-\frac{\partial}{\partial\epsilon}\varphi_{i}\left(0,\epsilon\right)\right|\]
is small and $t_{1}$ is chosen such that \[
\left|\frac{1}{2t_{2}}\left(\varphi_{i}\left(t_{1},\epsilon+t_{2}\right)-\varphi_{i}\left(t_{1},\epsilon-t_{2}\right)\right)-\frac{1}{2t_{2}}\left(\varphi_{i}\left(0,\epsilon+t_{2}\right)-\varphi_{i}\left(0,\epsilon-t_{2}\right)\right)\right|\]
is small, then $\left|\frac{\partial}{\partial\epsilon}\varphi_{i}\left(t_{1},\tilde{\epsilon}\right)-\frac{\partial}{\partial\epsilon}\varphi_{i}\left(0,\epsilon\right)\right|$
is small. Taking $t_{2}\rightarrow0$ and $t_{1}\rightarrow0$, we
have $\left(\left(0,\epsilon\right),\frac{\partial}{\partial\epsilon}\varphi_{i}\left(0,\epsilon\right)\right)\in\cl\left(\gph\left(\frac{\partial}{\partial\epsilon}\varphi_{i}\right)\right)$
as desired. 

Recall that the graph $\gph\left(\frac{\partial}{\partial\epsilon}\varphi_{i}\right)$
is taken corresponding to the domain $(0,\bar{t}]\times\left[\hat{\epsilon}_{1},\hat{\epsilon}_{2}\right]$,
and is a manifold of dimension $2$ in $\mathbb{R}^{3}$. Its boundary
is of dimension $1$ \cite[Proposition 3.16]{C02}, so the intersection
of $\cl\left(\gph\left(\frac{\partial}{\partial\epsilon}\varphi_{i}\right)\right)$
with $\left\{ 0\right\} \times\left[\hat{\epsilon}_{1},\hat{\epsilon}_{2}\right]\times\mathbb{R}$
is of dimension $1$ as well, and is homeomorphic to a closed line
segment. There cannot be an interval $\left[\tilde{\epsilon}_{1},\tilde{\epsilon}_{2}\right]\subset\left[\hat{\epsilon}_{1},\hat{\epsilon}_{2}\right]$
on which $\cl\left(\gph\left(\frac{\partial}{\partial\epsilon}\varphi_{i}\right)\right)\cap\left\{ 0\right\} \times\left\{ \epsilon\right\} \times\mathbb{R}$
has more than one value for all $\epsilon\in\left[\tilde{\epsilon}_{1},\tilde{\epsilon}_{2}\right]$
because by appealing to Proposition \ref{pro:continuity-proposition},
this implies that the dimension cannot be $1$. We note however that
it is possible that there exists an $\bar{\epsilon}\in\left[\hat{\epsilon}_{1},\hat{\epsilon}_{2}\right]$
such that $\cl\left(\gph\left(\frac{\partial}{\partial\epsilon}\varphi_{i}\right)\right)\cap\left\{ 0\right\} \times\left\{ \bar{\epsilon}\right\} \times\mathbb{R}$
is a $1$-dimensional line segment. This can only happen for only
finitely many $\bar{\epsilon}\in\left[\hat{\epsilon}_{1},\hat{\epsilon}_{2}\right]$
due to semi-algebraicity.

In any case, we can contract the interval $\left[\hat{\epsilon}_{1},\hat{\epsilon}_{2}\right]$
if necessary so that $\cl\left(\gph\left(\frac{\partial}{\partial\epsilon}\varphi_{i}\right)\right)\cap\left\{ 0\right\} \times\left\{ \epsilon\right\} \times\mathbb{R}$
is a single point for all $\epsilon\in\left[\hat{\epsilon}_{1},\hat{\epsilon}_{2}\right]$.
This means that for any $\left(t,\tilde{\epsilon}\right)\rightarrow\left(0,\epsilon\right)$,
we have $\frac{\partial}{\partial\epsilon}\varphi_{i}\left(t,\tilde{\epsilon}\right)\rightarrow\frac{\partial}{\partial\epsilon}\varphi_{i}\left(0,\epsilon\right)$,
establishing the continuity of $\frac{\partial}{\partial\epsilon}\varphi_{i}\left(\cdot,\cdot\right)$
on $\left[0,\bar{t}\right]\times\left[\hat{\epsilon}_{1},\hat{\epsilon}_{2}\right]$.
A reparametrization allows us to assume that $\bar{t}=1$, and we
are done.
\end{proof}

\section{\label{sec:quad-exa}Quadratic examples}

In this section, we show how the robust regularization can be calculated
for quadratic examples, which are more-or-less standard in the spirit
of \cite{Boy94,BtN}. We write $A\succeq0$ for a real symmetric matrix
$A$ if $A$ is positive semidefinite.

\begin{thm}
(Euclidean norm) For any real $m\times n$ matrix $A$ and vector
$b\in\R^{m}$, consider the function $g:\R^{n}\rt\R$ defined by \[
g(x)=\| Ax+b\|_{2},\]
Then the following properties are equivalent for any point $(x,t)\in\R^{n}\times\R$: 

\begin{enumerate} \item[{\rm (i)}]$t\geq\bar{g}_{\epsilon}\left(x\right)$

\item[{\rm (ii)}]there exists a real $\mu$ such that \begin{eqnarray*}
\left[\begin{array}{ccc}
tI_{m} & Ax+b & \e A\\
(Ax+b)^{T} & t-\mu & 0\\
\e A^{T} & 0 & \mu I_{n}\end{array}\right] & \succeq & 0.\end{eqnarray*}
\end{enumerate}
\end{thm}
\begin{proof}
Applying \cite[Thm 4.5.60]{BtN} shows $t\geq\bar{g}_{\e}(x)$ holds
if and only if there exist real $s$ and $\mu$ satisfying \begin{eqnarray*}
t-s & \geq & 0\\
\left[\begin{array}{ccc}
sI_{m} & Ax+b & \e A\\
(Ax+b)^{T} & s-\mu & 0\\
\e A^{T} & 0 & \mu I_{n}\end{array}\right] & \succeq & 0.\end{eqnarray*}
 and the result now follows immediately. 
\end{proof}
\noindent Since the matrix in property (ii) above is an affine function
of the variables $x$, $t$ and $\mu$, it follows that the robust
regularization $\bar{g}_{\e}$ is {}``semidefinite-representable'',
in the language of \cite{BtN}. This result allows us to use $\bar{g}_{\e}$
in building tractable representations of convex optimization problems
as semidefinite programs.

An easy consequence of the above result is a representation for the
robust regularization of any strictly convex quadratic function.

\begin{cor}
(quadratics) For any real positive definite $n$-by-$n$ matrix $H$,
vector $c\in\R^{n}$, and scalar $d$, consider the function $h:\R^{n}\rt\R$
defined by \[
h(x)=x^{T}Hx+2c^{T}x+d.\]
Then the following properties are equivalent for any point $(x,t)\in\R^{n}\times\R$: 

\begin{enumerate} \item[{\rm (i)}]$t\ge\bar{h}_{\e}(x)$;

\item[{\rm (ii)}]there exist reals $s$ and $\mu$ such that \begin{eqnarray*}
t-s^{2}+c^{T}H^{-1}c-d & \geq & 0\\
\left[\begin{array}{ccc}
sI_{n} & H^{1/2}x+H^{1/2}c & \e H^{1/2}\\
(H^{1/2}x+H^{-1/2}c)^{T} & s-\mu & 0\\
\e H^{1/2} & 0 & \mu I_{n}\end{array}\right] & \succeq & 0.\end{eqnarray*}

\end{enumerate}
\end{cor}
\begin{proof}
Clearly $t\geq\bar{h}_{\e}(x)$ if and only if \[
\Vert y-x\Vert_{2}\leq\e~~\Rightarrow~~\Vert H^{1/2}y+H^{-1/2}c\Vert_{2}^{2}\leq t-d+c^{T}H^{-1}c.\]
 This property in turn is equivalent to the existence of a real $s$
satisfying \begin{eqnarray*}
s^{2} & \leq & t-d+c^{T}H^{-1}c~~\mbox{and}\\
\Vert y-x\Vert_{2}\leq\e & \Rightarrow & \Vert H^{1/2}y+H^{-1/2}c\Vert_{2}\leq s,\end{eqnarray*}
 and the result now follows from the preceding theorem. 
\end{proof}
\noindent Since the quadratic inequality \[
t-s^{2}+c^{T}H^{-1}c-d\geq0\]
 is semidefinite-representable, so is the robust regularization $\bar{h}_{\e}$.

\section{\label{sec:The-condition}$1$-peaceful sets}

In this section, we prove that $X\subset\mathbb{R}^{n}$ is nearly
radial implies $X$ is $1$-peaceful using the Mordukhovich Criterion
\cite[Theorem 9.40]{RW98}, which relates the Lipschitz modulus of
set-valued maps to normal cones of its graph. The next section discusses
further properties of nearly radial sets and how they are common in
analysis.

The Mordukhovich Criterion requires the domain of the set-valued map
to be $\mathbb{R}^{n}$, so we recall the map $\tilde{\Phi}_{\epsilon}:\mathbb{R}^{n}\rightrightarrows\mathbb{R}^{n}$
by $\tilde{\Phi}_{\epsilon}\left(x\right)=\mathbb{B}_{\epsilon}\left(x\right)\cap X$.
Recall that $\tilde{\Phi}_{\epsilon}|_{X}=\Phi_{\epsilon}$ and $\lip\Phi_{\epsilon}\left(x\right)\leq\lip\tilde{\Phi}_{\epsilon}\left(x\right)$
for all $x\in X$. Let us recall the definitions of normal cones,
the Aubin property and the graphical modulus. 

\begin{defn}
\cite[Definition 6.3]{RW98} Let $X\subset\mathbb{R}^{n}$ and $\bar{x}\in X$.
A vector $v$ is normal to $X$ at $\bar{x}$ in the regular sense,
or a \emph{regular normal}, written $v\in\hat{N}_{X}\left(\bar{x}\right)$,
if\[
\left\langle v,x-\bar{x}\right\rangle \leq o\left(\left|x-\bar{x}\right|\right)\mbox{ for }x\in X.\]
It is normal to $X$ at $\bar{x}$ in the general sense, or simply
a \emph{normal} vector, written $v\in N_{X}\left(\bar{x}\right)$,
if there are sequences $x^{\nu}\xrightarrow[X]{}\bar{x}$ and $v^{\nu}\xrightarrow[X]{}v$
with $v^{\nu}\in\hat{N}_{X}\left(x^{\nu}\right)$.
\begin{defn}
\cite[Definition 9.36]{RW98} For $X\subset\mathbb{R}^{n}$, a mapping
$S:X\rightrightarrows\mathbb{R}^{m}$ has the \emph{Aubin property
at $\bar{x}$ for $\bar{u}$}, where $\bar{x}\in X$ and $\bar{u}\in S\left(\bar{x}\right)$,
if $\gph S$ is locally closed at $\left(\bar{x},\bar{u}\right)$
and there are neighborhoods $V$ of $\bar{x}$ and $W$ of $\bar{u}$
such that \[
S\left(x^{\prime}\right)\cap W\subset S\left(x\right)+\kappa\left|x^{\prime}-x\right|\mathbb{B}\mbox{ for all }x,x^{\prime}\in X\cap V.\]
The \emph{graphical modulus of $S$ at $\bar{x}$ for $\bar{u}$}
is \begin{eqnarray*}
\lip S\left(\bar{x}\mid\bar{u}\right) & := & \inf\{\kappa\mid\mbox{There are neighbourhoods }\\
 &  & \qquad\qquad V\mbox{ of }\bar{x}\mbox{, }W\mbox{ of }\bar{u}\mbox{ such that }\\
 &  & \qquad\qquad S\left(x^{\prime}\right)\cap W\subset S\left(x\right)+\kappa\left|x^{\prime}-x\right|\mathbb{B}\\
 &  & \qquad\qquad\mbox{ for all }x,x^{\prime}\in X\cap V\}.\end{eqnarray*}
If $S$ is single-valued at $\bar{x}$, then in keeping with the notation
of $\lipx$ in Definition \ref{def:calm-and-lip}, we write $\lip S\left(\bar{x}\right)$
instead of $\lip S\left(\bar{x}\mid S\left(\bar{x}\right)\right)$.
Note that this equals $\lip S\left(\bar{x}\right)$ if $S$ is continuous
at $\bar{x}$. $\diamond$
\end{defn}
\end{defn}
A set-valued map $S$ is \emph{locally compact} around $\bar{x}$
if there exist a neighborhood $V$ of $\bar{x}$ and a compact set
$C\subset Y$ such that $S\left(V\right)\subset C$. This is equivalent
to $S\left(V\right)$ being a bounded set, which is the case when
$S$ is outer semicontinuous and $S\left(\bar{x}\right)$ is bounded.
If $S$ is outer semicontinuous and locally compact at $\bar{x}$,
then by \cite[Theorem 1.42]{Mor06}, the Lipschitz modulus and the
Aubin property are related by \[
\lipi S\left(\bar{x}\right)=\max_{\bar{u}\in S\left(\bar{x}\right)}\left\{ \lip S\left(\bar{x}\mid\bar{u}\right)\right\} .\]
In finite dimensions, we need $S\left(\bar{x}\right)$ to be bounded
and $S$ to be outer semicontinuous for the formula above to hold. 

Here is a lemma on convex cones. 

\begin{lem}
\label{lem:simple-12.22}Given any two convex cones $C_{1}$ and $C_{2}$
polar to each other and any vector $x$, we have \[
\left(d\left(x,C_{1}\right)\right)^{2}+\left(d\left(x,C_{2}\right)\right)^{2}=\left\Vert x\right\Vert ^{2}\]

\end{lem}
\begin{proof}
This is a simple consequence of \cite[Exercise 12.22]{RW98}
\end{proof}
We now present our result on the relation between $1$-peaceful sets
and nearly radial sets.

\begin{thm}
\label{pro:nearly-radial-Lipschitz-1}If $X$ is nearly radial at
$\bar{x}$, then $X$ is $1$-peaceful at $\bar{x}$. The converse
holds if $X$ is subdifferentially regular for all points in a neighborhood
around $\bar{x}$.
\end{thm}
\begin{proof}
The graph of $\tilde{\Phi}_{\epsilon}$ is the intersection of $\mathbb{R}^{n}\times X$
and the set $D\subset\mathbb{R}^{n}\times\mathbb{R}^{n}$ defined
by \[
D:=\left\{ \left(x,y\right)\mid\left\Vert x-y\right\Vert \leq\epsilon\right\} .\]
By applying a rule on the normal cones of products of sets \cite[Proposition 6.41]{RW98},
we infer that $N_{\mathbb{R}^{n}\times X}\left(x,y\right)=\left\{ \mathbf{0}\right\} \times N_{X}\left(y\right)$.
Define the real valued function $g_{0}:\mathbb{R}^{n}\times\mathbb{R}^{n}\rightarrow\mathbb{R}_{+}$
by $g_{0}\left(x,y\right):=\frac{1}{2}\left\Vert x-y\right\Vert ^{2}$.
Then the gradient of $g_{0}$ is $\nabla g_{0}\left(x,y\right)=\left(x-y,y-x\right)$. 

From this point, we assume that $\left\Vert x-y\right\Vert =\epsilon$.
The normal cone of $D$ at $\left(x,y\right)$ is $N_{D}\left(x,y\right)=\mathbb{R}_{+}\left\{ \left(x-y,y-x\right)\right\} $
using \cite[Exercise 6.7]{RW98}. On applying a rule on the normal
cones of intersections \cite[Theorem 6.42]{RW98}, we get \begin{equation}
N_{\scriptsize\gph\tilde{\Phi}_{\epsilon}}\left(x,y\right)\subset\left(\left\{ \mathbf{0}\right\} \times N_{X}\left(y\right)\right)+\mathbb{R}_{+}\left\{ \left(x-y,y-x\right)\right\} .\label{eq:n-r-L-1-1}\end{equation}

Furthermore, if $X$ is subdifferentially regular at $y$, the above
set inclusion is an equation. By the Mordukhovich criterion \cite[Theorem 9.40]{RW98},
$\tilde{\Phi}_{\epsilon}$ has the Aubin Property at $\left(x,y\right)$
if and only if the graphical modulus $\lip\tilde{\Phi}_{\epsilon}\left(x\mid y\right)$
is finite. It can be calculated by appealing to the formulas for the
coderivative $D^{*}$ \cite[Definition 8.33]{RW98} and outer norm
$\left|\cdot\right|^{+}$ \cite[Section 9D]{RW98} below.  \begin{eqnarray}
\lip\tilde{\Phi}_{\epsilon}\left(x\mid y\right) & = & \left|D^{*}\tilde{\Phi}_{\epsilon}\left(x\mid y\right)\right|^{+}\mbox{ (by [}\ref{the:RW98}\mbox{, Theorem 9.40])}\nonumber \\
 & = & \sup_{w\in\mathbb{B}}\sup_{z\in D^{*}\tilde{\Phi}_{\epsilon}\left(w\right)}\left\Vert z\right\Vert \mbox{ (by [}\ref{the:RW98}\mbox{, Section 9D])}\nonumber \\
 & = & \sup\left\{ \left\Vert z\right\Vert \mid\left(w,z\right)\in\gph D^{*}\tilde{\Phi}_{\epsilon},\left\Vert w\right\Vert \leq1\right\} \nonumber \\
 & = & \sup\left\{ \left\Vert z\right\Vert \mid\left(-z,w\right)\in N_{\scriptsize\gph\tilde{\Phi}_{\epsilon}}\left(x,y\right),\left\Vert w\right\Vert \leq1\right\} \nonumber \\
 &  & \mbox{ (by [}\ref{the:RW98}\mbox{, Definition 8.33])}\nonumber \\
 & \leq & \sup\{\left\Vert z\right\Vert \mid\left(-z,w\right)\in\left(\left\{ \mathbf{0}\right\} \times N_{X}\left(y\right)\right)\label{eq:n-r-L-1-2}\\
 &  & \quad\quad+\mathbb{R}_{+}\left\{ \left(x-y,y-x\right)\right\} ,\left\Vert w\right\Vert \leq1.\nonumber \end{eqnarray}
We can assume that $z=y-x$ with a rescaling, and $w=y-x+v$ for some
$v\in N_{X}\left(y\right)$. Since $\left(\left\{ \mathbf{0}\right\} \times N_{X}\left(y\right)\right)+\mathbb{R}_{+}\left\{ \left(x-y,y-x\right)\right\} $
is positively homogeneous set, we could find the supremum of $\frac{\left\Vert z\right\Vert }{\left\Vert w\right\Vert }$
in the same set and the formula reduces to \begin{eqnarray}
\lip\tilde{\Phi}_{\epsilon}\left(x\mid y\right) & \leq & \sup_{v\in N_{X}\left(y\right)}\frac{\left\Vert y-x\right\Vert }{\left\Vert y-x+v\right\Vert }\nonumber \\
 & = & \sup_{v\in N_{X}\left(y\right)}\frac{\left\Vert x-y\right\Vert }{\left\Vert \left(x-y\right)-v\right\Vert }\nonumber \\
 & = & \frac{\left\Vert x-y\right\Vert }{d\left(x-y,N_{X}\left(y\right)\right)}.\label{eq:n-r-1-3}\end{eqnarray}

For a fixed $x\neq y$, say $\bar{x}$, we have $1/\lip\tilde{\Phi}_{\epsilon}\left(\bar{x}\mid y\right)\geq\frac{d\left(\bar{x}-y,N_{X}\left(y\right)\right)}{\left\Vert \bar{x}-y\right\Vert }$.
First, we prove that for any open set $W$ about $\bar{x}$, we have
\begin{equation}
\inf_{{y\in W\cap X\atop y\neq\bar{x}}}\frac{d\left(\bar{x}-y,N_{X}\left(y\right)\right)}{\left\Vert \bar{x}-y\right\Vert }=\inf_{{y\in W\cap X\atop y\neq\bar{x}}}\frac{d\left(\bar{x}-y,\hat{N}_{X}\left(y\right)\right)}{\left\Vert \bar{x}-y\right\Vert }.\label{eq:two-infs}\end{equation}
It is clear that {}``$\leq$'' holds because $\hat{N}_{X}\left(y\right)\subset N_{X}\left(y\right)$,
so we proceed to prove the other inequality. Consider $d\left(\bar{x}-y,N_{X}\left(y\right)\right)$.
Let $v\in P_{N_{X}\left(y\right)}\left(\bar{x}-y\right)$, the projection
of $\left(\bar{x}-y\right)$ onto $N_{X}\left(y\right)$. Then $v\in N_{X}\left(y\right)$,
and so there exists $y_{i}\rightarrow y$, with $y_{i}\in W\cap X$,
and $v_{i}\rightarrow v$ such that $v_{i}\in\hat{N}_{X}\left(y_{i}\right)$.
So \begin{eqnarray*}
d\left(\bar{x}-y,N_{X}\left(y\right)\right) & = & d\left(\bar{x}-y,\mathbb{R}_{+}\left(v\right)\right)\\
 & = & \lim_{i\rightarrow\infty}d\left(\bar{x}-y,\mathbb{R}_{+}\left(v_{i}\right)\right)\\
 & = & \lim_{i\rightarrow\infty}d\left(\bar{x}-y_{i},\mathbb{R}_{+}\left(v_{i}\right)\right)\\
 & \geq & \limsup_{i\rightarrow\infty}d\left(\bar{x}-y_{i},\hat{N}_{X}\left(y_{i}\right)\right)\\
\Rightarrow\frac{d\left(\bar{x}-y,N_{X}\left(y\right)\right)}{\left\Vert \bar{x}-y\right\Vert } & \geq & \limsup_{i\rightarrow\infty}\frac{d\left(\bar{x}-y_{i},\hat{N}_{X}\left(y_{i}\right)\right)}{\left\Vert \bar{x}-y_{i}\right\Vert }.\end{eqnarray*}
Thus equation \ref{eq:two-infs} holds. Therefore \[
\liminf_{y\rightarrow\bar{x}}\frac{d\left(\bar{x}-y,\hat{N}_{X}\left(y\right)\right)}{\left\Vert \bar{x}-y\right\Vert }\geq1\mbox{ implies }\limsup_{y\rightarrow\bar{x}}\lip\tilde{\Phi}_{\left\Vert \bar{x}-y\right\Vert }\left(\bar{x}\mid y\right)\leq1,\]
so we may now consider only regular normal cones.

By Lemma \ref{lem:simple-12.22}, we deduce the following: \[
d\left(\bar{x}-y,\hat{N}_{X}\left(y\right)\right)^{2}+d\left(\bar{x}-y,\hat{N}_{X}\left(y\right)^{*}\right)^{2}=\left\Vert \bar{x}-y\right\Vert ^{2}\mbox{ for }y\in X.\]
Since $T_{X}\left(y\right)^{*}=\hat{N}_{X}\left(y\right)$ always
\cite[Theorem 6.28(a)]{RW98}, we apply Lemma \ref{lem:simple-12.22}
and get \[
d\left(\bar{x}-y,\hat{N}_{X}\left(y\right)\right)^{2}+d\left(\bar{x}-y,T_{X}\left(y\right)^{**}\right)^{2}=\left\Vert \bar{x}-y\right\Vert ^{2}\mbox{ for }y\in X.\]
As $T_{X}\left(y\right)\subset T_{X}\left(y\right)^{**}$ \cite[Corollary 6.21]{RW98},
this implies that\begin{equation}
d\left(\bar{x}-y,\hat{N}_{X}\left(y\right)\right)^{2}+d\left(\bar{x}-y,T_{X}\left(y\right)\right)^{2}\geq\left\Vert \bar{x}-y\right\Vert ^{2}\mbox{ for }y\in X.\label{eq:inequality-3}\end{equation}
Note that if $X$ is nearly radial at $\bar{x}$, then $\frac{1}{\left\Vert \bar{x}-y\right\Vert }d\left(\bar{x}-y,T_{X}\left(y\right)\right)\rightarrow0$
as $\epsilon=\left\Vert \bar{x}-y\right\Vert \downarrow0$, $y\in X$.
This means that \[
1/\lip\tilde{\Phi}_{\left\Vert \bar{x}-y\right\Vert }\left(\bar{x}\mid y\right)\geq\frac{1}{\left\Vert \bar{x}-y\right\Vert }d\left(\bar{x}-y,\hat{N}_{X}\left(y\right)\right)\rightarrow1,\]
so \[
\limsup_{y\xrightarrow[X]{}\bar{x},y\neq\bar{x}}\lip\tilde{\Phi}_{\|\bar{x}-y\|}\left(\bar{x}\mid y\right)\leq1,\]
where $y\xrightarrow[X]{}\bar{x}$ means $y\in X$ and $y\rightarrow\bar{x}$.

Recall that $\tilde{\Phi}_{\epsilon}$ has closed graph, and hence
it is outer semicontinuous \cite[Theorem 5.7(a)]{RW98}. It is also
locally bounded, so \[
\lipi\tilde{\Phi}_{\epsilon}\left(\bar{x}\right)=\max_{y\in S_{\epsilon}\left(\bar{x}\right)}\lip\tilde{\Phi}_{\epsilon}\left(\bar{x}\mid y\right)\]
by \cite[Theorem 1.42]{Mor06}. This gives us $\limsup_{\epsilon\rightarrow0}\lipi\tilde{\Phi}_{\epsilon}\left(\bar{x}\right)\leq1$,
or $X$ is $1$-peaceful at $\bar{x}$, as needed.

If we assume that $X$ is regular in a neighborhood of $\bar{x}$,
then Formula \eqref{eq:inequality-3} is an equation. Furthermore,
\eqref{eq:n-r-L-1-1}, \eqref{eq:n-r-L-1-2} and \eqref{eq:n-r-1-3}
are all equations. Thus if $\lim_{\epsilon\rightarrow0}\lipi\tilde{\Phi}_{\epsilon}\left(\bar{x}\right)=1$,
then \[
\frac{1}{\left\Vert \bar{x}-y\right\Vert }d\left(\bar{x}-y,\hat{N}_{X}\left(y\right)\right)=1/\lip\tilde{\Phi}_{\left\Vert \bar{x}-y\right\Vert }\left(\bar{x}\mid y\right)\rightarrow1\mbox{ as }y\xrightarrow[X]{}\bar{x},\, y\neq\bar{x}.\]
and we have $\frac{1}{\left\Vert \bar{x}-y\right\Vert }d\left(\bar{x}-y,T_{X}\left(y\right)\right)\rightarrow0$
as $y\xrightarrow[X]{}\bar{x}$ and $y\neq\bar{x}$, which means that
$X$ is nearly radial at $\bar{x}$.
\end{proof}
Finally, $1-$peaceful sets are interesting in robust regularization
for another reason. The Lipschitz modulus of the robust regularization
over $1$-peaceful sets have Lipschitz modulus bounded above by that
of the original function, as the following result shows. 

\begin{prop}
\label{pro:Lipschitz_upper_bound}If $X$ is $1$-peaceful and $F:X\rightarrow\mathbb{R}^{n}$
is locally Lipschitz at $\bar{x}$, then \[
\limsup_{\epsilon\rightarrow0}\lipi F_{\epsilon}\left(\bar{x}\right)\leq\lip F\left(\bar{x}\right).\]

\end{prop}
\begin{proof}
We use a set-valued chain rule \cite[Exercise 10.39]{RW98}. Recall
the formula $F_{\epsilon}=\left(F\circ\tilde{\Phi}_{\epsilon}\right)\mid_{X}$.
The mapping $\left(x,u\right)\mapsto\tilde{\Phi}_{\epsilon}\left(x\right)\cap F^{-1}\left(u\right)$
is locally bounded because the map $x\mapsto\tilde{\Phi}_{\epsilon}\left(x\right)$
is locally bounded. Thus\[
\lipi F_{\epsilon}\left(\bar{x}\right)\leq\lipi\tilde{\Phi}_{\epsilon}\left(\bar{x}\right)\cdot\max_{x\in\tilde{\Phi}_{\epsilon}\left(\bar{x}\right)}\lip F\left(x\right).\]
By Theorem \ref{pro:nearly-radial-Lipschitz-1}, $\lim_{\epsilon\rightarrow0}\lipi\tilde{\Phi}_{\epsilon}\left(\bar{x}\right)\leq1$.
Also, since $\lip F:\mathbb{\mathbb{R}}^{n}\rightarrow\mathbb{R}_{+}$
is upper semicontinuous, $\limsup_{\epsilon\rightarrow0}\max_{x\in\tilde{\Phi}_{\epsilon}\left(\bar{x}\right)}\lip F\left(x\right)\leq\lip F\left(\bar{x}\right)$.
Taking limits to both sides gives us what we need.
\end{proof}

\section{Nearly radial sets}

As highlighted in Section \ref{sec:The-condition}, nearly radial
sets are $1$-peaceful. In this section, we study the properties
of nearly radial sets and give examples of nearly radial sets to illustrate
their abundance in analysis. 

We contrast the definition of nearly radial sets given before Proposition
\ref{pro:inf-lip-regularization} with a stronger property introduced
by \cite{Sha94}, which is the uniform version of the same idea. This
idea was called \emph{$o(1)$-convexity} in \cite{Sha94}. 

\begin{defn}
(nearly convex sets) A set $\A\subset\mathbb{R}^{n}$ is \emph{nearly
convex} at a point $\bar{x}\in\A$ if \[
\mbox{dist}\left(y,x+T_{\A}\left(x\right)\right)=o\left(\| x-y\|\right)\mbox{ as }x,y\rightarrow\bar{x}\mbox{ in }\A\]
The set $\A$ is \emph{nearly convex} if it is nearly convex at every
point $\A$. $\diamond$
\end{defn}
Clearly if a set is nearly convex at a point, then it is nearly radial
there, but the class of nearly radial sets is considerably broader.
For example, the set \[
\A=\{ x\in\mathbb{R}^{2}:x_{1}x_{2}=0\}\]
 is nearly radial at the origin but not nearly convex there, since
as $n\rightarrow\infty$ the points $x_{n}=(n^{-1},0)$ and $y_{n}=(0,n^{-1})$
approach the origin in $\A$ and yet \[
\mbox{dist}(y_{n},x_{n}+T_{\A}(x_{n}))=n^{-1}\ne o(\Vert x_{n}-y_{n}\Vert).\]

It is immediate that convex sets are nearly convex, and hence nearly
radial. A straightforward exercise shows that smooth manifolds are
also nearly convex, and hence again nearly radial. These observations
are both special cases of the following result, rather analogous to
\cite[Theorem 2.2]{Sha94}. A set $\A\subset\mathbb{R}^{n}$ is \emph{amenable}
\cite[Section 10F]{RW98} at a point $\bar{x}\in\A$ if there is an
open neighborhood $V$ of $\bar{x}$, a ${\mathcal{C}}^{1}$ mapping
$F:V\rightarrow\mathbb{R}^{m}$, and a closed convex set $D\subset\mathbb{R}^{m}$,
such that \begin{eqnarray}
 &  & \A\cap V=\left\{ x\in V:F\left(x\right)\in D\right\} \nonumber \\
 & \mbox{ and } & N_{D}\left(F\left(\bar{x}\right)\right)\cap N\left(\nabla F\left(\bar{x}\right)^{*}\right)=\left\{ \mathbf{0}\right\} ,\label{eq:6.2}\end{eqnarray}
where $N_{D}(\cdot)$ denotes the normal cone to $D$, and $N(\cdot)$
denotes null space. If in fact $F$ is ${\mathcal{C}}^{2}$ then we
call $\A$ \emph{strongly amenable} \cite[Definition 10.23]{RW98}
at $\bar{x}$.

\begin{thm}
(amenable implies nearly radial)\label{thm:amenable=3D>nearly-radial}
Suppose the set $\A\subset\mathbb{R}^{n}$ is amenable at the point
$\bar{x}\in\A$.  Then $\A$ is nearly convex (and hence nearly radial)
at $\bar{x}$. 
\end{thm}
\begin{proof}
Since $\A$ is amenable at $\bar{x}$, we can suppose property (\ref{eq:6.2})
holds. Suppose without loss of generality $\bar{x}=\mathbf{0}$, and
consider a sequences of points $x_{r},y_{r}\rightarrow\mathbf{0}$
in the set $\A\cap V$. We want to show \[
\mbox{dist}(y_{r},x_{r}+T_{\A}(x_{r}))=o(\Vert x_{r}-y_{r}\Vert).\]
 Without loss of generality we can suppose $x_{r}\ne y_{r}$ for all
$r$, and denote the unit vectors $\Vert x_{r}-y_{r}\Vert^{-1}(x_{r}-y_{r})$
by $z_{r}$. We want to prove \[
d_{r}=\min\{\Vert w+z_{r}\Vert:w\in T_{\A}(x_{r})\}\rightarrow0.\]
 The unique minimizer $w_{r}\in T_{\A}(x_{r})$ in the above projection
problem satisfies\begin{eqnarray*}
d_{r} & = & \Vert w_{r}+z_{r}\Vert\\
w_{r}+z_{r} & \in & -N_{\A}(x_{r})=-\nabla F(x_{r})^{*}N_{D}(F(x_{r}))\\
\left\langle w_{r},w_{r}+z_{r}\right\rangle  & = & 0,\end{eqnarray*}
 by \cite[Exercise 10.26(d)]{RW98}. Choose vectors $u_{r}\in-N_{D}(F(x_{r}))$
such that \[
w_{r}+z_{r}=\nabla F(x_{r})^{*}u_{r}.\]

We next observe that the sequence of vectors $\{ u_{r}\}$ is bounded.
Otherwise, we could choose a subsequence $\{ u_{r'}\}$ satisfying
$\Vert u_{r'}\Vert\rightarrow\infty$, and then any limit point of
the sequence of unit vectors $\{\Vert u_{r'}\Vert^{-1}u_{r'}\}$ must
lie in the set $-N_{D}(F(\mathbf{0}))\cap N(\nabla F(\mathbf{0})^{*})$,
contradicting property (\ref{eq:6.2}).

We now have \begin{eqnarray*}
0 & \leq & d_{r}^{2}\,=\,\left\langle z_{r},\nabla F(x_{r})^{*}u_{r}\right\rangle \,=\,\left\langle \nabla F(x_{r})z_{r},u_{r}\right\rangle \\
 & = & \left\langle \nabla F(x_{r})z_{r}-\Vert x_{r}-y_{r}\Vert^{-1}[F(x_{r})-F(y_{r})],u_{r}\right\rangle \\
 &  & \qquad\qquad+\left\langle \Vert x_{r}-y_{r}\Vert^{-1}[F(x_{r})-F(y_{r})],u_{r}\right\rangle .\end{eqnarray*}
 The first term converges to zero, using the smoothness of the mapping
$F$ and the boundedness of the sequence $\{ u_{r}\}$. On the other
hand, since the set $D$ is convex, we have $F(y_{r})-F(x_{r})\in T_{D}(F(x_{r}))$,
and $u_{r}\in-N_{D}(F(x_{r}))$ by assumption, so the second term
is nonpositive, and the result follows. 
\end{proof}
It is worth comparing these notions to a property that is slightly
stronger still: \emph{prox-regularity} (in the terminology of \cite[Section 13F]{RW98}),
or \emph{$O(2)$-convexity} \cite{Sha94}. 

\begin{defn}
(prox-regular sets) A set $\A\subset\mathbb{R}^{n}$ is \emph{prox-regular}
at a point $\bar{x}\in\A$ if \[
\mbox{dist}\left(y,x+T_{\A}\left(x\right)\right)=O\left(\| x-y\|^{2}\right)\mbox{ as }x,y\rightarrow\bar{x}\mbox{ in }\A.\diamond\]

\end{defn}
Theorem \ref{thm:amenable=3D>nearly-radial} (amenable implies nearly
radial) is analogous to the fact that strong amenability implies prox-regularity
\cite[Proposition 13.32]{RW98} (and also to \cite[Proposition 2.3]{Sha94}).

The class of nearly radial sets is very broad, as the following easy
result (which fails for nearly convex sets) emphasizes.

\begin{prop}
\label{pro:unions}(unions) If the sets $\A_{1},\A_{2},\ldots,\A_{n}$
are each nearly radial at the point $\bar{x}\in\cap_{j}\A_{j}$, then
so is the union $\cup_{j}\A_{j}$. 
\end{prop}
\begin{proof}
If the result fails, there is a sequence of points $x_{r}\rightarrow\bar{x}$
in $\cup_{j}\A_{j}$ and real $\epsilon>0$ such that 

\begin{equation}
\mbox{dist}\left(\frac{\bar{x}-x_{r}}{\Vert\bar{x}-x_{r}\Vert},T_{\cup_{j}\A_{j}}(x_{r})\right)\geq\epsilon\mbox{ for all }r.\label{eq:6.5}\end{equation}
By taking a subsequence, we can suppose that there is an index $i$
such that $x_{r}\in\A_{i}$ for all $r$. But then we know \[
\mbox{dist}\left(\frac{\bar{x}-x_{r}}{\Vert\bar{x}-x_{r}\Vert},T_{\A_{i}}(x_{r})\right)\rightarrow0,\]
 which contradicts inequality (\ref{eq:6.5}), since $T_{\A_{i}}(x_{r})\subset T_{\cup_{j}\A_{j}}(x_{r})$. 
\end{proof}
A key concept in variational analysis is the idea of Clarke regularity
(see for example \cite{Cla83,Cla98,RW98}). We make no essential use
of this concept in our development, but it is worth remarking on the
relationship (or lack of it) between the nearly radial property and
Clarke regularity. Note first that nearly radial sets need not be
Clarke regular: the union of the two coordinate axes in $\mathbb{R}^{2}$
is nearly radial at the origin, for example, but it is not Clarke
regular there.

On the other hand, Clarke regular sets need not be nearly radial. 

\begin{example}
\label{exa:not-nearly-radial}Consider the function $f:\mathbb{R}\rightarrow\mathbb{R}$
defined by {\small \[
f(x)=\left\{ \begin{array}{ll}
2^{-n}-2^{-n-1}(2-2^{n+1}|x|)^{1+2^{-n}} & \mbox{if}~2^{-n-1}\leq|x|\leq2^{-n}~(n\in\mathbb{N})\\
0 & \mbox{if}~x=0.\end{array}\right.\]
}The function $f$ is even, and its graph consists of concave segments
on each interval $x\in[2^{-n-1},2^{-n}]$, passing through the point
$2^{-n}(1,1)$ with left derivative zero, and through the point $2^{-n-1}(1,1)$
with right derivative $1+2^{-n}$. A routine calculation now shows
that this function is everywhere regular, and hence its epigraph $\mbox{\textrm{epi}}\, f$
is everywhere Clarke regular. However, $\mbox{\textrm{epi}}\, f$
is not nearly radial at the origin. To see this, observe that for
each $n\in\mathbb{N}$, if we consider the sequence $x_{n}=2^{-n}(1,1)\rightarrow(0,0)$,
then we have \[
T_{\mbox{\scriptsize epi}\, f}(x_{n})=\left\{ (x,y):y\geq(1+2^{1-n})\max\{ x,0\}\right\} ,\]
 so \[
\mbox{dist}(0,x_{n}+T_{\mbox{\scriptsize epi}\, f}(x_{n}))=\frac{\Vert x_{n}\Vert}{\sqrt{2}},\]
 contradicting the definition of a nearly radial set. $\diamond$
\end{example}
This is yet another attractive property for semi-algebraic sets.

\begin{thm}
(semi-algebraic sets)\label{thm:semi-algebraic=3D>nearly-radial}
Semi-algebraic sets are nearly radial.
\end{thm}
\begin{proof}
Suppose the origin lies in a semi-algebraic set $\A\subset\mathbb{R}^{n}$.
We will show that $\A$ is nearly radial at the origin.

If the result fails, then there is a real $\delta>0$ and a sequence
of points $y_{r}\rightarrow\mathbf{0}$ in $\A$ such that \[
\left\Vert u+\frac{y_{r}}{\Vert y_{r}\Vert}\right\Vert >\delta~~\mbox{for all}~u\in T_{\A}(y_{r}).\]
 Hence for each index $r$ there exists a real $\gamma_{r}>0$ such
that \[
\Big\Vert\frac{z-y_{r}}{\Vert z-y_{r}\Vert}+\frac{y_{r}}{\Vert y_{r}\Vert}\Big\Vert>\delta~~\mbox{for all}~z\in\A~\mbox{such that}~0<\Vert z-y_{r}\Vert<\gamma_{r}.\]
 Consequently, each point $y_{r}$ lies in the set \begin{eqnarray*}
\A_{0} & = & \Big\{ y\in\A\mid\exists\gamma>0~\mbox{so}~\Big\Vert\frac{z-y}{\Vert z-y\Vert}+\frac{y}{\Vert y\Vert}\Big\Vert>\delta\\
 &  & \qquad\qquad\qquad\forall z\in\A\setminus\{ y\}~\mbox{with}~\Vert z-y\Vert<\gamma\Big\},\end{eqnarray*}
 so $\mathbf{0}\in\cl\A_{0}$.

By quantifier elimination (see for example the discussion of the Tarski-Seidenberg
Theorem in \cite[p.~62]{Ben90}), the set $\A_{0}$ is semi-algebraic.
Hence the Curve Selection Lemma (see \cite[p.~98]{Ben90} and \cite{Mil68})
shows that there is a real-analytic path $p:[0,1]\rt\mathbb{R}^{n}$
such that $p(0)=\mathbf{0}$ and $p(t)\in\A_{0}$ for all $t\in(0,1]$.
For some positive integer $k$ and nonzero vector $g\in\mathbb{R}^{n}$
we have, for small $t>0$, \begin{eqnarray*}
p(t) & = & gt^{k}+O(t^{k+1})\\
p'(t) & = & kgt^{k-1}+O(t^{k}),\end{eqnarray*}
 and in particular both $p(t)$ and $p'(t)$ are nonzero. For any
such $t$ we know \[
\left\Vert \frac{z-p(t)}{\Vert z-p(t)\Vert}+\frac{p(t)}{\Vert p(t)\Vert}\right\Vert >\delta\]
 for any point $z\in\A\setminus\{ p(t)\}$ close to $p(t)$. Hence
for any real $s\neq t$ close to $t$ we have \[
\left\Vert \frac{p(s)-p(t)}{\Vert p(s)-p(t)\Vert}+\frac{p(t)}{\Vert p(t)\Vert}\right\Vert >\delta.\]
 Taking the limit as $s\uparrow t$ shows \[
\left\Vert \frac{p(t)}{\Vert p(t)\Vert}-\frac{p'(t)}{\Vert p'(t)\Vert}\right\Vert \geq\delta\]
 for all small $t>0$. But since \[
\lim_{t\downarrow0}\frac{p(t)}{\Vert p(t)\Vert}=\frac{g}{\Vert g\Vert}=\lim_{t\downarrow0}\frac{p'(t)}{\Vert p'(t)\Vert},\]
 this is a contradiction. 
\end{proof}
\noindent By contrast, semi-algebraic sets need not be nearly convex.
For example, the union of the two coordinate axes in $\mathbb{R}^{2}$
is semi-algebraic, but it is not nearly convex at the origin.

\begin{acknowledgement*}
Thanks to Jim Renegar for helpful discussions concerning Theorem \ref{thm:semi-algebraic=3D>nearly-radial}
(semi-algebraic sets), and to Jon Borwein for the statement of the
definition of nearly radial sets before Proposition \ref{pro:inf-lip-regularization}.
The first author's research was supported in part by U.S. NSF grant
DMS-080657.
\end{acknowledgement*}

\end{document}